\documentclass[12pt]{iopart}
\usepackage{setstack,iopams, amsthm, mathrsfs}
\usepackage{color}
\usepackage{graphicx}

\newtheorem{thm}{Theorem}[section]
\newtheorem{lem}[thm]{Lemma}
\newtheorem{prop}[thm]{Proposition}
\newtheorem{cor}[thm]{Corollary}

 \newtheorem{assumption}{Assumption}
 \newtheorem*{definition}{Definition}
 \newtheorem{rem}{Remark}
 \newtheorem{example}{Example}
 \newtheorem*{rem*}{Remark}
 \newtheorem*{thm*}{Theorem}
 \newtheorem{proposition}{Proposition}
 
 \newcommand{\BRP}{BRP}
\newcommand{\EIP}{EIP}

\begin{document}
\title[Posterior Consistency for Bayesian Inverse Problems]{Posterior Consistency for Bayesian Inverse Problems through Stability and Regression Results}

\author{Sebastian J. Vollmer \footnote{Present address: Department of Statistics, University of Oxford, 1 South Parks Road, Oxford OX1 3TG.}}

\address{Mathematics Institute, Zeeman Building, University of Warwick,
Coventry CV4 7AL, UK.}

\ead{Sebastian.Vollmer@stats.ox.ac.uk}

\global\long\def\data{y}
\global\long\def\smallBallAsymp{\rho}
\global\long\def\noise{\xi}
\global\long\def\input{a}
\global\long\def\truth{a^{\dagger}}
\global\long\def\prior{\mu_{0}}
\global\long\def\posterior{\mu^{\data_{n}}}
\global\long\def\priorCov{\mathcal{C}_{0}}
\global\long\def\nrOfObs{k}
\global\long\def\obsCov{\mathcal{C}_{1}}
\global\long\def\spInput{X}
\global\long\def\spResponse{Y}
\global\long\def\opObs{\mathcal{G}}
\global\long\def\betterTrace{\sigma_{0}}
\global\long\def\noiseMeasure{\mathbb{\mu_{\xi}}}
\global\long\def\nrOfObs{k}
\global\long\def\obsCov{\Gamma}
\global\long\def\pressure{p}
\global\long\def\pressureSpace{V}

\global\long\def\HilbertScalesInter{\lambda}
\global\long\def\HilbertScalesYoungs{\eta}
\global\long\def\nRadiusSnorm{\theta}
\global\long\def\deltarate{\kappa}
\global\long\def\priorSmoothness{s}
\global\long\def\truthSmoothness{\tau}
\global\long\def\priorTails{r}

\begin{abstract}
In the Bayesian approach, the a priori knowledge about the input of
a mathematical model is described via a probability measure. The joint
distribution of the unknown input and the data is then conditioned,
using Bayes\textquoteright{} formula, giving rise to the posterior
distribution on the unknown input. In this setting we prove posterior
consistency for nonlinear inverse problems: a sequence of data is
considered, with diminishing fluctuations around a single truth and
it is then of interest to show that the resulting sequence of posterior
measures arising from this sequence of data concentrates around the
truth used to generate the data. Posterior consistency justifies the
use of the Bayesian approach very much in the same way as error bounds
and convergence results for regularisation techniques do.
As a guiding example, we consider the inverse problem of reconstructing
the diffusion coefficient from noisy observations of the solution
to an elliptic PDE in divergence form. This problem is approached
by splitting the forward operator into the underlying continuum model
and a simpler observation operator based on the output of the model. 

In general, these splittings allow us to conclude posterior consistency
provided a deterministic stability result for the underlying inverse
problem and a posterior consistency result for the Bayesian regression
problem with the push-forward prior. 

Moreover, we prove posterior consistency for the Bayesian regression
problem based on the regularity, the tail behaviour and the small
ball probabilities of the prior.

\end{abstract}

\ams{35R30, 62C10, 62G20}
\submitto{\IP}
%Bayesian inverse problems, nonlinear problems, posterior consistency, Gaussian noise

\maketitle
\section{Introduction}
Many mathematical models used in science and technology contain parameters
for which a direct observation is very difficult. A good example is
subsurface geophysics. The aim in subsurface geophysics is the reconstruction
of subsurface properties such as density and permeability given measurements
on the surface.Using the laws of physics, these properties
can be used as parameters of a forward model mapping them to the
measurements which we subsequently call data.

Inverting such a relationship is non-trivial and lies in the focus
of the area of inverse problems. Classically, these parameters
are estimated by minimisation of a regularised least squares functional
which is based on the data output mismatch (Tikhonov).
The idea of this approach is to use optimisation techniques aiming
at parameters that produce nearly the same noiseless output as the
given noisy data while being not too irregular. However,
it is difficult to quantify how the noise in the data translates into
the uncertainty of the reconstructed parameters for this method. Uncertainty
quantification is much more straightforward in the Bayesian approach.
The basic idea of the Bayesian method is that not all parameter choices
are a priori equally likely. Instead, the parameters are artificially
treated as random variables by modelling their distribution using
a priori knowledge. This distribution is accordingly called the prior.
For a specific forward model and given the distribution of the observational
noise, the parameters and the data can be treated as jointly varying
random variables. Under mild conditions, the prior can then be updated
by conditioning the parameters on the data.

The posterior is one of the main tools for making inference about
the parameters. Possible estimates include approximation of the posterior
mean or the maximum a posteriori (MAP) estimator. Moreover,
it is possible to quantify the uncertainty of the reconstructed parameter
by posterior variance or posterior probability of a set around for
example an estimate of the parameters under consideration.

The main focus of this article lies on posterior consistency which
quantifies the quality of the resulting posterior in a thought experiment.
As for any evaluation for an approach to inverse problems, an identical
twin experiment is performed, that is for a fixed set of parametersand
artificial data is generated. It is conceivable to expect that, under
appropriate conditions, the posterior concentrates around this set
of 'true' parameters. Results of this type are called posterior consistency.
It justifies the Bayesian method by establishing that this method
recovers the 'true' parameters sometimes with a specific rate.

So far, there are only posterior consistency results available for
linear forward models and mainly Gaussian priors \cite{knapik2011bayesian,2012arXiv1203.5753A,ray2012bayesian,florens2012regularized}.
In this article, we prove posterior consistency of nonlinear inverse
problems with explicit bounds on the rate. The main idea behind our
posterior consistency results is to use stability properties of the
deterministic inverse problem to reduce posterior consistency of a
nonlinear inverse problem to posterior consistency of a Bayesian non-parametric
regression problem. Our guiding example is the inverse problem of
reconstructing the diffusion coefficient from measurements of the
pressure. More precisely, we assume that the relation between the
diffusion coefficient $a$ and the pressure $p$ satisfies the following
partial differential equation (PDE) with Dirichlet boundary conditions
\begin{eqnarray}\label{eq:postConEllipticForward}
\cases{ -\nabla \cdot(\input\nabla\pressure) =f(x)\quad\mbox{ in }D\\ \quad\pressure = 0 \quad \qquad\qquad\quad\mbox{on }\partial D}
\end{eqnarray}
where $D$ is a bounded smooth domain in $\mathbb{R}^{d}.$ For this
guiding example the required stability results are due to \cite{MR628945}.
However, our methods are generally applicable to inverse problems
with deterministic stability results. These are often available in
the literature because they are also needed for convergence results
of the Tikhonov regularisation (consider for example Theorem 10.4.
in \cite{inverseProblem}). Finally, we complete our reasoning by
proving appropriate posterior consistency results for the corresponding
Bayesian non-parametric regression problem.

\subsection*{Structure of this Article}

In Section \ref{sec:Preliminary}, we both review preliminary material
and give a detailed exposition of our main ideas, steps and results. In Section \ref{sec:General-Results}, we provide novel posterior
consistency results for Bayesian non-parametric regression. In order
to evaluate the rate for the regression problem, we compare our rates
to those for Gaussian priors for which optimal rates are known. These
results are needed in order to obtain posterior consistency for the
elliptic inverse problem in Section \ref{sec:Results-for-Elliptic}.
We obtain explicit rates for priors based on a series expansion with
uniformly distributed coefficients. In Section \ref{sec:Conclusion},
we draw a conclusion and mention other inverse problems to which this
approach is applicable. The appendix contains a detailed summary of
relevant technical tools such as Gaussian measures and Hilbert scales
which are used in the proofs of our main results.

\ack
The author would like to thank Professor Martin Hairer, Professor
Andrew Stuart, Dr. Hendrik Weber and Sergios Agapiou for helpful discussions.
SJV is grateful for the support of an ERC scholarship.

\section{Preliminaries and Exposition to Posterior Consistency for Nonlinear
Inverse Problems\label{sec:Preliminary}}

Our crucial idea for proving posterior consistency for a nonlinear
Bayesian inverse problem is the use of stability results which allow
us to break it down to posterior consistency of a Bayesian regression
problem. Because the proofs are quite technical, it is worth becoming
familiar with the outline of our main ideas first. Therefore this
section is intended to motivate, review and summarise our investigation
of posterior consistency for a nonlinear inverse problem leaving technical
details to the Sections \ref{sec:General-Results} and \ref{sec:Results-for-Elliptic}.
For the convenience of the reader we also repeat the relevant material
on Bayesian inverse problems in Section \ref{sub:ReviewBayesian}
without proofs, thus making our exposition self-contained. In Section
\ref{sub:Posterior-Consistency}, we precisely define posterior consistency
in this setting and place it within the literature. \textbf{S}ubsequently,
we introduce an elliptic inverse problem as guiding example for which
we apply our method using stability results from \cite{MR628945}.

Finally, we conclude our exposition by giving a general abstract theorem
of posterior consistency for nonlinear inverse problems with stability
results in Section \ref{sub:Consistency-through-Stability}.

\subsection{Summary of the Bayesian Approach to Inverse Problems on Hilbert Spaces\label{sub:ReviewBayesian}}

The key idea of Bayesian inverse problems is to model the input $\input\in X$
of a mathematical model, for example the initial condition of a PDE,
as random variable with distribution $\prior(d\input)$ based on a
priori knowledge. This distribution is called the prior which is updated
based on the observed data $y$. The resulting distribution $\mu^{\data}$
is called posterior and lies in the focus of the Bayesian approach. 

We assume that the data is modelled as 
\begin{equation}
y=\mathcal{G}(\input)+\noise\label{eq:IP}
\end{equation}
with $\mathcal{G}$ being the forward operator, a mapping between
the Hilbert spaces $X$ and $Y$, and with the observational noise
$\noise$. The aim of the inverse problem is the reconstruction of
$\input$ given the data $\data$. Because $\mathcal{G}$ might be
non-injective and $\xi$ is unknown, the problem is not exactly solvable
as stated.\textbf{ }If the distribution of the noise $\noise$ is
known, then $\input$ and $\data$ can be treated as jointly varying
random variables. Under mild assumptions on the prior, the distribution
of the noise and the forward operator, there exists a conditional
probability measure on $\input$, called the posterior $\mu^{y}$.
It is an update of the prior using the data and models the a posteriori
uncertainty. Therefore it can be viewed as the solution to the inverse
problem itself. In this way it is possible to obtain different explanations
of the data corresponding to different modes of the posterior.

In this article, we assume that the law of the observational noise
$\mu_{\noise}=\mathcal{N}(0,\obsCov)$ is a mean-zero Gaussian with
covariance $\obsCov$. In this case Bayes' rule can be generalised
for any $\mathcal{G}$ mapping into a finite dimensional space $Y$.
It follows that

\begin{eqnarray}
\hspace{-2cm}\frac{d\mu^{y}}{d\mu_{0}}(\input) & \propto\exp\left(-\frac{1}{2}\bigl\Vert\opObs(\input)-y\bigr\Vert_{\obsCov}^{2}\right)\propto\exp\left(-\frac{1}{2}\bigl\Vert\opObs(\input)\bigr\Vert_{\obsCov}^{2}+\left\langle y,\opObs(\input)\right\rangle _{\obsCov}-\bigl\Vert y\bigr\Vert_{\obsCov}^{2}\right)\label{eq:GeneralPosterior}\\
 & \propto\exp\left(-\frac{1}{2}\bigl\Vert\opObs(\input)\bigr\Vert_{\obsCov}^{2}+\left\langle y,\opObs(\input)\right\rangle _{\obsCov}\right).\nonumber 
\end{eqnarray}

By $\left\Vert \cdot\right\Vert _{\obsCov}$ we denote the norm of
the Cameron-Martin space $(H_{\noiseMeasure},\left\langle \cdot,\cdot\right\rangle _{\obsCov})$
of $\mu_{\noise}$ that is the closure of $Y$ with respect to $\left\langle \cdot,\cdot\right\rangle _{\obsCov}=\left\langle \obsCov^{-1}\cdot,\cdot\right\rangle $
(see \ref{sec:Notation} for more details). 
A proper derivation of Equation (\ref{eq:GeneralPosterior}),
including the fact that its last line is also valid for functional
data, and an appropriate introduction to Bayesian inverse problems
can be found in \cite{MR2652785} and \cite{stuartchinanotes}. All
in all the Bayesian approach can be summarised as %

\begin{equation}
\begin{array}{ll}
\mbox{Prior} & \input\sim\mu_{0} \\[10pt]
\mbox{Noise} & \noise\:\sim\mathcal{N}(0,\Gamma) \\[10pt]
\mbox{Posterior} & \frac{d\mu^{y}}{d\mu_{0}}\propto\exp\left(-\frac{1}{2}\bigl\Vert\opObs(\input)\bigr\Vert_{\obsCov}^{2}+\left\langle y,\opObs(\input)\right\rangle _{\obsCov}\right).
\end{array}\label{eq:bayesInv}
\end{equation}
As one can see in this example, the posterior can usually only be
expressed implicitly as an unnormalised density with respect to the
prior. Thus, in order to estimate the input parameters or perform
inference using the posterior, it has to be probed using either
\begin{itemize}
\item sampling methods, such as MCMC which aim at generating draws from
the posterior or 
\item variational methods for determining the location of an infinitesimal
ball with maximal posterior probability.
\end{itemize}
The second approach is also called the maximum a posteriori probability
(MAP) estimator. It can be viewed as an extension to many classical
methods for inverse problems. For example, it can be linked to the
$L^{2}$-Tikhonov regularisation by considering a Gaussian prior and
noise \cite{2013dashtiMap}. This relates the choice of norms in the
Tikhonov regularisation to the choice of the covariance of the prior
and the noise.

These regularisation techniques can be justified by convergence results.
Similarly, inference methods based on the posterior can be justified
by posterior consistency, a concept which we introduce in the next
section.

\subsection{Posterior Consistency for Bayesian Inverse Problems\label{sub:Posterior-Consistency}}

As for any approach to inverse problems, the Bayesian method can be
evaluated by considering an identical twin experiment. Therefore a
fixed input $\input^{\dagger}$, called the \textit{'truth'}, is considered
and data is generated using a sequence of forward models 
\[
y_{n}=\mathcal{G}_{n}(\truth)+\noise_{n}
\]
which might correspond to the increasing amount of data or diminishing
noise.   For each $n$  we denote the posterior corresponding to the prior $\prior$,  the noise distribution $\mu_{\xi_n}$ and the forward operator $\mathcal{G}_n$ by  $\mu^{y}_{n}$. Under appropriate assumptions,  the posterior $\mu_n^y $ is well-defined for $y=\mathcal{G}(\input)+\noise$  given by Bayes' rule in Equation (\ref{eq:bayesInv}) for $\prior$-a.e. $\input$ and $\mu_{\xi_n}$-a.e. $\noise_n$  (c.f. \cite{stuartchinanotes}).  This Bayes' rule does not give rise to a well-defined measure for arbitrary $y$. However, we will pose assumptions such that the normalising constant in the Bayes' rule  will be bounded above and below for every $\truth$ belonging to a particular set and $\mu_{\xi_n}$-a.e. $y=y_{n}=\mathcal{G}_{n}(\truth)+\noise_{n}$. We will denote these posteriors by $\mu^{y_n}$. This sequence of inverse problems is called posterior consistent if the posteriors $\posterior$ concentrate around the 'truth' $\truth$.
We quantify the concentration by the posterior probability assigned
to the ball $B_{\epsilon}^{d}$. Here $B_{\epsilon}^{d}$ denotes
a ball of radius $\epsilon$ with respect to a metric $d$. 

In the following we define this concept precisely and place it within
the literature before closing this section by relating posterior consistency
to small ball probabilities for the prior.

\begin{definition}
\textbf{\label{def:Posterior-consistency}}(Analogue to \cite{MR1790007})
A sequence of Bayesian inverse problems $(\prior,\mathcal{G}_{n},\mathcal{L}(\xi_{n}))$
is posterior consistent for $\truth$ with rate $\epsilon_{n}\downarrow0$
and with respect to a metric $d$ if for 
\[
y_{n}=\mathcal{G}_{n}(\truth)+\noise_{n},
\]
there exists a constant $M$ and a sequence $l_{n}\rightarrow1$ such
that 
\begin{equation}
\mathbb{P}_{\noise_{n}}\left(\mu^{y_{n}}\mbox{\ensuremath{\left(B_{M\epsilon_{n}}^{d}(\truth)\right)}}\geq l_{n}\right)\rightarrow1.\label{eq:posteriorConsistency}
\end{equation}
We simply say that $(\prior,\mathcal{G}_{n},\mathcal{L}(\xi_{n}))$
is posterior consistent if the above holds for any fixed constant
$\epsilon_{n}=\epsilon>0$. 
\end{definition}

Two important special cases of this definition are

\begin{itemize}
\item posterior consistency in the small noise limit: 
\[
\mathcal{L}\left(\xi_{n}\right)=\mathcal{L}\left(\frac{1}{\sqrt{n}}\xi\right)\mbox{ and }\mathcal{G}_{n}=\mathcal{G}
\]

\item posterior consistency in the large data limit: 
\[
\mathcal{L}\left(\xi_{n}\right)=\otimes_{i=1}^{n}\mathcal{L}\left(\xi\right)\mbox{ and \ensuremath{\mathcal{G}_{n}}=\ensuremath{\prod_{i=1}^{n}}}\mathcal{G}^{i}=(\mathcal{G}^{1},\dots,\mathcal{G}^{n}).
\]

\end{itemize}
In the above formulation $\mathcal{G}^{i}$ corresponds to different
measurements while $\mathcal{L}(\cdot)$ denotes the law of a random
variable.\textbf{}

There exists a variety of results for\textbf{ }posterior consistency
and inconsistency for statistical problems. Two important examples
are the identification of a distribution from (often i.i.d.) samples
or density estimation \cite{postConIncon,MR1790007,MR2418663,BayesianNonBook}.
The former is concerned with considering a prior on a set of probability
distributions and the resulting posterior based on $n$ samples of
one of these probability distributions. In \cite{doob1949application},
Doob proved that if a countable collection of samples almost surely
allows the identification of the generating distribution, then the
posterior is consistent for almost every probability distribution
with respect to the prior. This very general result is not completely
satisfactory because it does not provide a rate and the interest may
lie in showing posterior consistency for every possible truth in a
certain class. Moreover, some surprisingly simple examples of posterior
inconsistency have been provided for example by considering distributions
on $\mathbb{N}$ \cite{MR0158483}. The necessary bounds for posterior
consistency (c.f. Equation (\ref{eq:posteriorConsistency})) can be
obtained using the existence of appropriate statistical tests which
are due to bounds on entropy numbers. These methods are used in a
series of articles, for example in \cite{MR1790007,shen2001rates,MR2418663,vaartBook}.
This idea has also recently been applied to the Bayesian approach
to linear inverse problems in \cite{ray2012bayesian}.

In general, posterior consistency for infinite dimensional inverse
problems has mostly been studied for linear inverse problems in the
small noise limit where the prior is either a sieve prior, a Gaussian
or a wavelet expansion with uniform distributed coefficients \cite{knapik2011bayesian,2012arXiv1203.5753A,ray2012bayesian,florens2012regularized}.
Except for \cite{ray2012bayesian}, all these articles exploit the
explicit structure of the posterior in the conjugate Gaussian setting,
that means that we have a Gaussian prior as well as a Gaussian posterior. 

In contrast, we consider general priors, general forward operators
and Gaussian noise in this article. Usually, the posterior has a density
with respect to the prior as in Equation (\ref{eq:bayesInv}). However,
it is possible to provide examples where both the prior and posterior
are Gaussian but not absolutely continuous. This can be achieved using
for example Proposition 3.3 in \cite{severeillposedbay}. 

Subsequently, we assume that the posterior has a density with respect
to the prior implying that the posterior probability of a set is zero
whenever the prior probability of this set is zero. Therefore it is
necessary that $\truth$ is in the support of the prior giving rise
to the following definition.
\begin{definition}
\label{ass:nonDegenerateDef}The support of a measure $\mu$ in a
metric space $\left(X,d\right)$ is given by

\[
\mbox{supp}_{d}(\mu)=\left\{ x\Big\vert\mu\left(B_{\epsilon}^{d}(x)>0\:\forall\epsilon>0\right)\right\} .
\]
\end{definition}
It is natural to expect that the posterior consistency rate depends
on the behaviour of $\prior\left(B_{\epsilon}^{d}(a^{\dagger})\right)$
as $\epsilon\rightarrow0$. Asymptotics of this type are called small
ball probabilities. We recommend \cite{smallballsurv} as a good survey
and refer the reader to \cite{smallBallBib} for an up-to-date list
of references. In this article, we consider algebraic rates of posterior
consistency, that means we take $\epsilon_{n}=n^{-\kappa}$ in Definition
\ref{def:Posterior-consistency}. In order to establish these rates
of posterior consistency, we consider small ball asymptotics of the
following form

\[
\log(\prior(B_{\epsilon}^{d}(\truth))\succsim-\epsilon^{-\smallBallAsymp},
\]
where $\rho>0$ and with the notation as in Appendix \ref{sec:Notation}. 

Both posterior consistency and the contraction rate depend on properties
of the prior. This suggests that we should choose a prior with favourable
posterior consistency properties. From a dogmatic point of view the
prior is only supposed to be chosen to match the subjective a priori
knowledge. In practice priors are often picked based on their computational
performance whereas some of their parameters are adapted to represent
the subjective knowledge. An example for this is the choice of the
base measure and the intensity for a Dirichlet process \cite{BayesianNonBook}. 

Finally, we would like to conclude this Section by mentioning that
it has been shown in \cite{MR829555} that posterior consistency is
equivalent to the property that the posteriors corresponding to two
different priors merge. The yet unpublished book \cite{vaartBook}
contains a more detailed discussion about the justification of posterior
consistency studies for dogmatic Bayesians.

\subsection{An Elliptic Inverse Problem as an Application of our Theory\label{sub:IntroElliptic}}

The aim of this section is to set up the elliptic inverse problem
for which we will prove posterior consistency (c.f. Section \ref{sub:Posterior-Consistency})
both in the small noise and the large data limit. In a second step
we describe the available stability results and how they can be used
to reduce the problem of posterior consistency of a nonlinear inverse
problem to that of a linear regression problem. We end this section
by stating a special case of our posterior consistency results in
Section \ref{sec:Results-for-Elliptic}. 

Our results do not only apply to this particular elliptic inverse
problem but to any nonlinear inverse problem with appropriate stability
results (c.f. Section \ref{sub:Consistency-through-Stability}). However,
the results for the elliptic inverse problem are of particular interest
because it is used in oil reservoir simulations and the reconstruction
of the groundwater flow \cite{yeh1986review,MR628945,hansen2012inverse}. 

The forward model corresponding to our elliptic inverse problem is
based on the relation between $\pressure$ and $\input$ given by
the elliptic PDE in \ref{eq:postConEllipticForward}.

We would like to highlight that the relation between $\input$ and
$\pressure$ is nonlinear. Under the following assumptions, the solution
operator $\pressure(x;\input)$ to the above PDE is well-defined \cite{MR1814364}.

\begin{assumption}
(\label{ass:ellipticForward}Forward conditions) Suppose that
\begin{enumerate}
\item $D$ is compact, satisfies the exterior sphere condition (see \cite{MR1814364})
and has a smooth boundary;
\item $a\in C^{1}(D)\cap C(\bar{D})$ and f is smooth in $\bar{D}$;
\item $\input>\input_{\mbox{min}}>0$ and $f>f_{\mbox{min}}>0$ in
Equation (\ref{eq:postConEllipticForward}).
\end{enumerate}
\end{assumption}
Under these assumptions, the regularity results from \cite{MR1814364}
yield the following forward stability result.
\begin{prop}
\label{prop:forwardStability}If $a_{1}$ and $a_{2}$ satisfy Assumption
\ref{ass:ellipticForward} and are elements of $C^{\alpha}$ for $\alpha\ge1$,
then
\begin{equation}
\left\Vert p(\cdot;a_{1})-p(\cdot;a_{2})\right\Vert _{C^{\alpha+1}}\leq M\left\Vert a_{1}-a_{2}\right\Vert _{C^{\alpha}}.\label{eq:forwardCont}
\end{equation}

\end{prop}
The inverse problem is concerned with the reconstruction of $a$ given
the data 
\[
y_{n}=\mathcal{G}_{n}(a)+\xi,
\]
which is related to $p$ in the following way. 
\begin{assumption}
\label{ass:forwardSplitting}The forward operator $\mathcal{G}$ can
be split into a composition of the solution operator $\pressure$
and an observation operator $\mathcal{O}$, that is 
\begin{equation}
\mathcal{G}_{n}(\input)=\mathcal{O}_{n}\left(p(\cdot;\input)\right).\label{eq:assForwardSplitting}
\end{equation}

\end{assumption}
The Bayesian approach to the Elliptic Inverse Problem (\EIP) summarises as
\[
\boxed{\begin{array}{cc}
\mbox{Model} & -\nabla\cdot(\input\nabla\pressure(\cdot;a))=f(x)\;\mbox{in }D,\quad\pressure=0\mbox{ on }\partial D\\[7pt]
\mbox{Prior } & \prior\mbox{ on }\input \\[7pt]
\mbox{Data} & \data=\opObs_{n}(\input)+\noise_{n}=\mathcal{O}_{n}(\pressure(\cdot,\input))+\noise_{n},\,\noise_{n}\sim\mathcal{N}(0,\Gamma_{n}) \\[7pt]
\mbox{Posterior} & \frac{d\mu^{n}}{d\prior}(\input)\propto\exp\left(-\frac{1}{2}\bigl\Vert\opObs_{n}(\input)\bigr\Vert_{\obsCov_{n}}^{2}+\left\langle y,\opObs(\input)\right\rangle _{\obsCov_{n}}\right).
\end{array}}\quad\left(\mbox{EIP}\right)
\]
A rigorous Bayesian formulation of this inverse problem, with log-Gaussian
priors and Besov priors has been given in \cite{UncertaintyElliptic}
and \cite{Con3} respectively. In \cite{Con4} the problem is considered
with a prior based on a series expansion with uniformly distributed
coefficients (see Section \ref{sub:Posterior-Consistency-Rate}).
In the same article, a generalised Polynomial Chaos (gPC) method is
derived in order to approximate posterior expectations.

We consider posterior consistency as set up in Definition \ref{def:Posterior-consistency}
in the following cases: 
\begin{itemize}
\item the small noise limit with $\mathcal{O}_{n}=\mbox{Id}$ corresponding
to a functional observation and an additive Gaussian random field
as noise such that 
\[
y_{n}=\pressure(\cdot;u)+\frac{1}{\sqrt{n}}\noise;
\]

\item the large data limit with $\mathcal{O}_{n}=\left(e_{x_{i}}\right)_{i=1}^{n}$
where $e_{x_{i}}$ are evaluations at $x_{i}\in D$. In this case
the data takes the form
\[
y_{n}=\left\{ \pressure(x_{i};\input)\right\} _{i=1}^{n}+\xi_{n}.
\]

\end{itemize}

Posterior consistency in both cases are based on a stability result
which can be derived by taking $\input$ as the unknown in Equation
(\ref{eq:postConEllipticForward}). This leads to the following hyperbolic
PDE
\begin{equation}
-\nabla a\cdot\nabla p-a\Delta p=f.\label{eq:hyperbolicPDE}
\end{equation}

Imposing Assumption \ref{ass:ellipticForward}, it has been established
that there exists a unique solution $\input$ to this PDE without
any additional boundary conditions:
\begin{proposition}[Corollary 2 on page 220 in \cite{MR628945}] \label{prop:Richter}Suppose
$\pressure$ arises as a solution to Equation (\ref{eq:postConEllipticForward})
with $\input$ as diffusion coefficient satisfying Assumption \ref{ass:ellipticForward}.
Then Equation (\ref{eq:hyperbolicPDE}) is uniquely solvable for any
$f\in L^{\infty}(D)$ and $\input$ such that
\[
\bigl\Vert\input\bigr\Vert_{\infty}\leq D\left(\input_{\text{min}},f_{\text{min}},\bigl\Vert\nabla\input\bigr\Vert_{\infty}\right)\bigl\Vert\pressure\bigr\Vert_{\infty}.
\]
Moreover, if $a_{1}$ and $a_{2}$ satisfy these assumptions, then

\textup{
\[
\left\Vert \input_{1}-\input_{2}\right\Vert _{\infty}\leq M\left\Vert \input_{1}\right\Vert _{C^{1}}\cdot\left\Vert \pressure(\cdot,\input_{1})-\pressure(\cdot,\input_{2})\right\Vert _{C^{2}}.
\]
}
\end{proposition}

The stability result above and a change of variables (Theorem \ref{thm:IndpendentOfDiscription})
implies 

\[
\mu^{y_{n}}\bigl(B_{\epsilon}^{L^{\infty}}(\truth)\bigr)=\tilde{\mu}^{y_{n}}\bigl(p(B_{\epsilon}^{L^{\infty}}(\truth)\bigr)\ge\tilde{\mu}^{y_{n}}\left(B_{\frac{\epsilon}{M}}^{C^{2}}(\pressure^{\dagger})\right).
\]
This statement reduces posterior consistency of the \EIP\ in $L^{\infty}$
to posterior consistency of the following Bayesian Regression Problem
(\BRP) in $C^{2}$  
\[
\boxed{\begin{array}{cc}
\text{Prior } & \tilde{\mu}_{0}=\pressure_{\star}\prior\text{ on }\pressure \\[7pt]
\text{Data} & \data=\mathcal{O}_{n}(\pressure)+\noise_{n},\,\noise_{n}\sim\mathcal{N}\left(0,\Gamma_{n}\right) \\[7pt]
\text{Posterior} & \frac{d\tilde{\mu}^{y_{n}}}{d\tilde{\prior}}(\pressure)\propto\exp\left(-\frac{1}{2}\bigl\Vert\mathcal{O}(\pressure)\bigr\Vert_{\obsCov_{n}}^{2}+\left\langle y,\mathcal{O}(\pressure)\right\rangle _{\obsCov_{n}}\right) \\[7pt]
 & \text{with }\mathcal{O}_{n}=\text{Id}\text{ or }\mathcal{O}_{n}=\left(e_{x_{i}}\right)_{i=1}^{n}
\end{array}}\quad\left(\text{BRP}\right)
\]
where $p$ is now treated as an variable, that is  the prior and
the posterior are now formulated on the pressure space. Moreover,
$p_{\star}\prior$ denotes the push forward of the prior under $p$.
Note that for $\mathcal{O}_{n}=\text{Id}$ the \BRP\ can also be
viewed as the simplest linear inverse problem.

The required posterior consistency results for the \BRP\ can be derived
from those in Section \ref{sec:General-Results} using interpolation
inequalities. In this way we obtain posterior consistency results
in Section \ref{sec:Results-for-Elliptic} a special case of which
is the following theorem:

\begin{thm}[\ref{thm:ellipticNoRate}]
Suppose that the prior $\prior$ satisfies 
\begin{eqnarray*}
a(x)\ge\lambda>0\quad\forall x\in D\text{ and }\bigl\Vert a\bigr\Vert_{C^{\alpha}}\leq\Lambda\qquad\mbox{ for }\prior\mbox{-a.e. } a  \mbox{ and for } \alpha >1
\end{eqnarray*}
Let the noise be given by $\xi\sim\mathcal{N}\left(\mbox{0,(-\ensuremath{\Delta_{\text{Dirichlet}})^{-r}}}\right)$.
If $\alpha>r+\frac{d}{2}-2$ and $\alpha>r-1$, then (\EIP) is posterior
consistent for any $\truth\in\text{supp}_{C^{\alpha}}\prior$ in the
small noise limit with respect to the $C^{\tilde{\alpha}}$-norm for
any $\tilde{\alpha}<\alpha$.
\end{thm}

This approach is not limited to the \EIP\  as the following section
shows.

\subsection{Posterior Consistency through Stability Results\label{sub:Consistency-through-Stability}}

In Section \ref{sub:IntroElliptic}, we present our main idea, that
is the reduction of the problem of posterior consistency of the \EIP\ 
to that of the \BRP. The main ingredients of this reduction are the
stability result that was summarised in Proposition \ref{prop:Richter}
and the posterior consistency results for the \BRP. This approach
is not limited to the \EIP\  but it is applicable to any inverse
problem for which appropriate stability results are available. This
is the case for many inverse problems such as the inverse scattering
problem in \cite{kuchment2012radon} or the Calderon problem in \cite{alessandrini1988stable}. We would
like to point out that these stability results are also crucial for
proving the convergence of regularisation methods (see Theorem 10.4
in \cite{inverseProblem}).

\begin{thm}
Suppose $\mathcal{G}_{n}=\mathcal{O}_{n}\circ G$ with $G:(X,\Vert\cdot\Vert_{X})\rightarrow(Y,\Vert\cdot\Vert_{Y})$
and $\mathcal{O}_{n}:(Y,\Vert\cdot\Vert_{Y})\rightarrow(Z,\Vert\cdot\Vert_{Z})$.
Moreover, we assume that 
\begin{itemize}
\item there exists a stability result of the form 
\begin{eqnarray*}
\Vert a_{1}-a_{2}\Vert_{X} & \leq b(\Vert G(a_{1})-G(a_{2})\Vert_{Y})\\
 & \mbox{where }b:\mathbb{R}^{+}\rightarrow\mathbb{R}^{+}\mbox{is increasing and},\, b(0)=0;
\end{eqnarray*}

\item the sequence of Bayesian inverse problems $(G_{\star}\prior,\mathcal{O}_{n},\mathcal{L}(\xi_{n}))$
is posterior consistent with respect to $\Vert\cdot\Vert_{Y}$ for
all $p^{\dagger}\in A$ with rate $\epsilon_{n}$. 
\end{itemize}
Then $(\prior,\mathcal{G}_{n},\mathcal{L}(\xi_{n}))$ is posterior
consistent with respect to $\Vert\cdot\Vert_{X}$ for all $a^{\dagger}\in G^{-1}(A)$
with rate $b(\epsilon_{n}).$ \end{thm}
\begin{proof}
Using the notation of Section \ref{sub:IntroElliptic}, we denote
the posteriors for the Bayesian inverse problems $(\prior,\mathcal{G}_{n},\mathcal{L}(\xi_{n}))$
and $(G_{\star}\prior,\mathcal{O}_{n},\mathcal{L}(\xi_{n}))$ by $\posterior$
and $\tilde{\mu}^{y}$, respectively. Then a change of variables (c.f.
Theorem \ref{thm:IndpendentOfDiscription})
implies
\[
\mu^{y}\bigl(B_{b(\epsilon_{n})}^{X}(\truth)\bigr)\ge\tilde{\mu}^{y}\bigl(B_{\epsilon_{n}}^{Y}(G(\truth))\bigr).
\]
\end{proof}

\section{Posterior Consistency for Bayesian Regression\label{sec:General-Results}}

As described in the previous section, for many inverse problems posterior
consistency can be reduced to posterior consistency of a \BRP\ (c.f.
Section \ref{sub:Consistency-through-Stability}) using stability
results. Thus, with the results obtained in this section we may conclude
posterior consistency for apparently harder nonlinear inverse problems.
For the \EIP\ this is achieved by an application of the results in
Theorem \ref{thm:NoTail} and \ref{thm:largeData}. Because the derivation
of these two results is quite technical, we first give a summary and
we recommend the reader to become familiar with both theorems but
to skip the technical details on the first read. 

It is classical to model the response as
\[
\data_{n}=\mathcal{O}_{n}(\pressure)+\noise_{n}.
\]
In the following we consider two Bayesian regression models with
\begin{itemize}
\item $\mathcal{O}_{n}=\text{Id}$ and the noise is a Gaussian random field
that is scaled to zero like $\noise_{n}=n^{-\frac{1}{2}}\noise$ or 
\item $\mathcal{O}_{n}=\left(e_{x_{i}}\right)_{i=1}^{n}$ and $\mathcal{L}(\noise_{n})=\otimes_{i=1}^{n}\mathcal{N}\left(0,\sigma^{2}\right)$
corresponding to evaluations of a function with additive i.i.d. Gaussian
noise.
\end{itemize}
These models represent the large data and the small noise limit, respectively.\textbf{
}

We prove posterior consistency for both problems under weak assumptions
on the prior. This is necessary because the \BRP s resulting from
nonlinear inverse problems are usually only given in an implicit form.
For both cases we are able to obtain a rate assuming appropriate asymptotic
lower bounds on the small ball probabilities of the prior around $\truth$
(see Section \ref{sub:Posterior-Consistency}). Moreover, posterior
consistency with respect to stronger norms can be obtained using prior
or posterior regularity in combination with interpolation inequalities
which is the subject of Section \ref{sub:Convergence-in-Stronger}. 

For the large data limit, that is $\mathcal{O}_{n}=\left(e_{x_{i}}\right)_{i=1}^{n}$,
we obtain posterior consistency with respect to the $L^{\infty}$-norm
in Section \ref{sub:Point-wise-Observations-in}. We assume an almost
sure upper bound on a H�lder norm for the prior and an additional
condition on the locations of the observations. The latter is justified
by construction of a counterexample.

For the small noise limit, that is $\mathcal{O}_{n}=\text{Id}$, we
prove posterior consistency with respect to the Cameron-Martin norm
of the noise in Section \ref{sub:The-Small-Noise}. This norm corresponds
to the $\left\Vert \cdot\right\Vert _{1}$-norm in the Hilbert scale
with respect to the covariance operator $\Gamma$. Both the Cameron-Martin
norm and Hilbert scales are introduced in  \ref{sec:Notation}.
If an appropriate $\left\Vert \cdot\right\Vert _{s}$-norm is $\prior$-a.s.
bounded, we obtain an explicit rate of posterior consistency. Otherwise,
the rate is implicitly given as a low-dimensional optimisation problem.
However, the condition for mere posterior consistency takes a simple
form.

\begin{cor}
(See Corollary \ref{cor:postConSmallNoRate} for the case of general
noise)\\
Suppose that the noise is given by $\xi\sim\mathcal{N}(0,(-\Delta)^{-r})$
and $\prior\left(\exp(f\bigl\Vert p\bigr\Vert_{H^{s}}^{e})\right)<\infty$
for $s>r+\mbox{\ensuremath{\frac{d}{2}}}$ and $f>0$. Then the posterior
is consistent in $H^{r}$ for any $\truth\in\text{supp}_{\mathcal{H}^{r}}$
if $e$ and $\lambda=\frac{s-r-\frac{d}{2}}{s-r}$ satisfy the following
conditions {\small{
\[
\begin{array}{cc}
e>-1+\sqrt{8-8\lambda} & \text{if }\lambda\in\left[0,\frac{1}{2}\right] \\[7pt]
e>2-2\lambda & \text{if }\lambda\in\frac{1}{2},1.
\end{array}
\]
}}
{\small \par}\end{cor}
\begin{rem}
\label{Rem:NoRateCorollary-Gaussian}If the prior is Gaussian, then
the above inequality is satisfied because $e=2$ and the RHS is less
than $2$ for any $\lambda\in(0,1)$. Thus, the only remaining condition
is $s>r+\mbox{\ensuremath{\frac{d}{2}}}$.
\end{rem}

\begin{rem}
\label{Rem:NoRateCorollary-LogConcave}It is worth pointing out that
for the large class of log-concave measures it is known that $e\ge1$,
for details consult \cite{1974BorellLogConcave}. 
\end{rem}
In the statistics literature regression models are mainly concerned
with pointwise observations. Despite its name this is also true for
\textit{functional data analysis}\textbf{ }(see \cite{ferraty2006nonparametric}).
However, the regression problem associated with $\mathcal{O}_{n}=\text{Id}$
can be viewed as a particular linear inverse problem. As described
in the introduction, this has been studied for Gaussian priors in
\cite{knapik2011bayesian} and \cite{2012arXiv1203.5753A}. Although
our focus lies on establishing posterior consistency for general priors
and non-linear models, we also obtain rates which in the special case
of Gaussian priors are close to the optimal rates given in the references
above.

\subsection{The Small Noise Limit for Functional Response\label{sub:The-Small-Noise}}

In the following we study posterior consistency for a Bayesian regression
problem assuming that the data takes values in the Hilbert space $H$.
In particular we deal with the regression model 
\begin{equation}
y=\input+\frac{1}{\sqrt{n}}\xi\label{eq:postConInvers}
\end{equation}
with $y,\,\,\input$ and $\xi$ all being elements of $H$. Moreover,
we suppose that the observational noise $\noise$ is a Gaussian random
field $\mu_{\noise}=\mathcal{N}(0,\obsCov)$ on $H$ and we assume
that it satisfies the following assumption. 
\begin{assumption}
\label{ass:noiseTraceSumm}Suppose there is $\sigma_{0}\ge$ such
that $\Gamma^{\sigma}$ is trace-class for all $\sigma>\sigma_{0}$,
that is 
\[
\sum_{k=1}^{\infty}\lambda_{k}^{2\sigma}<\infty.
\]

\end{assumption}
Imposing this assumption, it becomes possible to quantify the regularity
of the observational noise in terms of the Hilbert scale defined with
respect to the covariance operator (c.f. \ref{sec:Notation}).
More precisely, this is possible due to Lemma \ref{lem:regularityNoise}.
from \cite{2012arXiv1203.5753A}.

The regression model in Equation (\ref{eq:postConInvers}) is a special
case of a general inverse problem as considered in Equation (\ref{eq:IP}).
Hence the corresponding posterior takes the following form (c.f. Equation
(\ref{eq:bayesInv})). 

\begin{equation}
\frac{d\mu^{\data}}{d\mu_{0}}=Z(n,\noise)\exp\left(-\frac{1}{2}n\left\Vert \input\right\Vert _{1}^{2}+n\left\langle \input,\data\right\rangle _{1}\right).\label{eq:BRPposterior}
\end{equation}
Assuming that the data takes values in the Hilbert space $H$, Equation
(\ref{eq:BRPposterior}) can simply be derived by an application of
the Cameron-Martin lemma in combination with the conditioning lemma
(Lemma 5.3 in \cite{nonlinearsampling}).
 We generate data for a fixed
'truth' $\truth$
\begin{equation}
\data=\truth+\frac{1}{\sqrt{n}}\noise.\label{eq:postConTruth}
\end{equation}
By changing the normalising constant, we may rewrite the posterior
in the following way 
\begin{equation}
\frac{d\mu^{\data}}{d\mu_{0}}=Z(n,\noise)\exp\left(-\frac{n}{2}\left\Vert \input-\truth\right\Vert _{1}^{2}+\sqrt{n}\left\langle \input-\truth,\noise\right\rangle _{1}\right).\label{eq:identityPosterior}
\end{equation}

The normalising constant is bounded above and below for $y_{n}=\mathcal{G}_{n}(\truth)+\noise_{n}$ for $\mu_{\xi_n}$-a.e. $\xi$. In fact, this holds  under weaker assumptions than needed for our results.
\begin{lem}
\label{lem:BRPnormalising}Suppose $\prior\left(\exp\left(f\left\Vert a\right\Vert _{s}^{e}\right)\right)<\infty$
for $s>1+\sigma_{0}$ and $e>\frac{2\sigma_{0}}{s-1+\sigma_{0}}$.
Then the normalising constant in Equation (\ref{eq:identityPosterior})
is bounded for $\mu_{\noise_n}$-a.s. and every $\truth\in\mathcal{H}^s$ above and away form zero.\end{lem}
\begin{proof}
See  \ref{sec:normal}. 
\end{proof}

The expression above suggests that the posterior concentrates in balls
around the truth in the Cameron-Martin norm. First, we make this fact
rigorous for priors which are a.s. uniformly bounded with respect
to the $\Vert\cdot\Vert_{s}$-norm. In a second step, we assume that
the prior has higher exponential moments. Considering Gaussian priors,
we show that our rate is close to the optimal rate obtained in \cite{knapik2011bayesian}.

\subsubsection{Posterior Consistency for Uniformly Bounded Priors}

The following theorem can be viewed as a preliminary step towards
Theorem \ref{thm:postConSmallGeneral} which contains our most general
posterior consistency result for the Bayesian regression problem in
the small noise limit. While containing our main ideas, the following
result also establishes an explicit rate for posterior consistency
which will be used for the \EIP\  in Section \ref{sec:Results-for-Elliptic}.
\begin{thm}
\label{thm:NoTail} Suppose that the noise satisfies Assumption \ref{ass:noiseTraceSumm}
and 
\begin{equation}
\bigl\Vert\input\bigr\Vert_{s}\leq U\;\prior\text{-a.s.}\label{eq:postConAsBdd}
\end{equation}
for $s>1+\sigma_{0}.$ If $\truth\in\text{supp}_{\mathcal{H}^{1}}(\prior)$
and {\textup{$\truth\in\mathcal{H}^{s}$,}}
then $\posterior$ is consistent in $\mathcal{H}^{1}$. Additionally,
if the following small ball asymptotic is satisfied 
\begin{equation}
\log(\prior(B_{\epsilon}^{1}(\truth))\succsim-\epsilon^{-\smallBallAsymp},\label{eq:yhmNoTailSmallBall}
\end{equation}
then this holds with rate $Mn^{-\kappa}$ for any $\kappa<\min\left\{ \frac{1}{2(2-\lambda)},\frac{1}{2+\rho}\right\} $
with $\HilbertScalesInter=\frac{s-1-\sigma_{0}}{s-1}$. \end{thm}
\begin{proof}
Our proof is based on the observation that posterior consistency is
implied by the existence of a sequence of subsets $S_{n}$ such that
$\noiseMeasure(S_{n})\rightarrow1$ and 
\begin{equation}
\sup_{\xi\in S_{n}}\frac{\posterior(B_{\epsilon n^{-\kappa}}^{1}(\truth)^{c})}{\posterior\left(B_{\epsilon n^{-\kappa}}^{1}(\truth)\right)}\rightarrow0\text{ for }n\rightarrow\infty\label{eq:RatioLimit}
\end{equation}
where $\data_{n}=\truth+\frac{1}{\sqrt{n}}\noise$. This implication
holds because $$\posterior(B_{\epsilon n^{-\kappa}}^{1}(\truth))+\posterior(B_{\epsilon n^{-\kappa}}^{1}(\truth)^{c})=1$$
and thus
\begin{equation}
\sup_{\xi\in S_{n}}\frac{\posterior(B_{\epsilon n^{-\kappa}}^{1}(\truth)^{c})}{\posterior\left(B_{\epsilon n^{-\kappa}}^{1}(\truth)\right)}\leq\delta\quad\Rightarrow\quad\frac{1}{1+\delta}\le\sup_{\xi\in S_{n}}\posterior\left(B_{\epsilon n^{-\kappa}}^{1}(\truth)\right)\label{eq:RatioLimitConseq}
\end{equation}
which together with $\noiseMeasure(S_{n})\rightarrow1$
implies posterior consistency, for details see Equation (\ref{eq:posteriorConsistency}).

Fix $\gamma>0$. Then $S_{n}=B_{K_{n}^{\prime}}^{1-\sigma_{0}-\gamma}(0)$
with $K_{n}^{\prime}\uparrow\infty$ as $n\rightarrow\infty$ sufficiently
slow. We notice that Lemma \ref{lem:regularityNoise} implies that
$\mathbb{P}_{\noise}(\noise\in B_{K_{n}^{\prime}}^{1-\sigma_{0}-\gamma}(0))\rightarrow1$
as $n\rightarrow\infty$. The remainder of the proof will be devoted
to showing that Equation (\ref{eq:RatioLimit}) holds. We bound $\left\langle \input-\truth,\noise\right\rangle _{1}$
by smoothing $\xi$ at the expense of $\input-\input^{\dagger}$
\begin{eqnarray*}
\left|\left\langle \input-\truth,\noise\right\rangle _{1}\right| & \leq & \left|\left\langle \obsCov^{-1+\frac{1-\sigma_{0}-\gamma}{2}}(\input-\truth),\obsCov^{\frac{\sigma_{0}-1+\gamma}{2}}\noise\right\rangle _{1}\right|\\
 & \leq & \left\Vert \input-\truth\right\Vert _{1+\sigma_{0}+\gamma}\left\Vert \xi\right\Vert _{1-\sigma_{0}-\gamma}\\
 & \leq & \left\Vert \input-\truth\right\Vert _{1+\sigma_{0}+\gamma}K_{n}^{\prime}\:\forall\xi\in B_{K_{n}^{\prime}}^{1-\sigma_{0}}(0).
\end{eqnarray*}
Interpolating between $\mathcal{H}^{1}$ and $\mathcal{H}^{s}$ for
$s$ (c.f. Lemma \ref{lem:HScaleInterpolation}) yields

\begin{equation}
\left|\left\langle \input-\truth,\noise\right\rangle _{1}\right|\leq K_{n}^{\prime}\left\Vert \input-\truth\right\Vert _{1}^{\HilbertScalesInter}\left\Vert \input-\truth\right\Vert _{s}^{1-\HilbertScalesInter}\leq K_{n}\left\Vert \input-\truth\right\Vert _{1}^{\HilbertScalesInter}\label{eq:postConNoTailInterp}
\end{equation}
with $\HilbertScalesInter=\frac{s-1-\sigma_{0}-\gamma}{s-1}$. An
application of Equation (\ref{eq:identityPosterior}) yields the following
upper bound
\begin{eqnarray}
\hspace{-2.5cm}\mu^{y}\bigl(B_{\frac{\epsilon}{n^{\kappa}}}^{1}(\truth)\bigr) \hspace{-0.4cm}&\ge \hspace{-0.1cm}Z(n,\xi)\underset{a\in B_{\frac{\epsilon}{2}n^{-\kappa}}^{1}}{\inf}\exp\left(-n\bigl\Vert a-a^{\dagger}\bigr\Vert_{1}^{2}-\sqrt{n}\left\langle a-a^{\dagger},\xi\right\rangle _{1}\right)\prior\hspace{-0.1cm}\left[B_{\frac{\epsilon}{2}n^{-\kappa}}^{1}\left(\truth\right)\right]\nonumber\\
 \hspace{-0.4cm}& \ge \hspace{-0.1cm}Z(n,\xi)\exp\left[-n^{1-2\kappa}\left[\frac{\epsilon\bigl\Vert a-a^{\dagger}\bigr\Vert_{1}}{2}\right]^{2}-K_{n}n^{\frac{1}{2}-\lambda\kappa}\left(\frac{\epsilon}{2}\right)^{\lambda}\right]\hspace{-0.1cm}\prior\hspace{-0.1cm}\left[B_{\frac{\epsilon}{2}n^{-\kappa}}^{1}\left(\truth\right)\right]\hspace{-0.1cm}.\label{eq:postConSmallNotTailLower}
\end{eqnarray}
Similarly, we obtain the following upper bound

\begin{eqnarray*}
\mu^{y}\left(B_{\epsilon n^{-\kappa}}^{1}\left(\truth\right)\right)\leq & Z(n,\xi)\underset{a\in B_{\epsilon n^{-\kappa}}^{1}(\truth)}{\sup}\exp\left(-n\bigl\Vert a-a^{\dagger}\bigr\Vert_{1}^{2}+K_{n}\sqrt{n}\left\Vert \input-\truth\right\Vert _{1}^{\HilbertScalesInter}\right).
\end{eqnarray*}
The expression in the exponential in Equation (\ref{eq:identityPosterior})
can be rewritten as a function $f(d)=-nd^{2}+K_{n}n^{\frac{1}{2}}d^{\lambda}$
of $d=\Vert\input-\truth\Vert$ which is decreasing on $[(K_{n}\lambda n^{-\frac{1}{2}}/2)^{\frac{1}{2-\lambda}},\infty)$.
If
\begin{equation}
-\frac{1}{2(2-\lambda)}<-\kappa,\label{eq:postConSmallNoTailRate1}
\end{equation}
then $\left\Vert \input-\truth\right\Vert _{1}^{2}\in[(K_{n}\lambda n^{-\frac{1}{2}}/2)^{\frac{1}{2-\lambda}},\infty)$
for $a\in B_{\epsilon n^{-\kappa}}^{1}(\truth)$ and $n$ large enough
leading to 
\begin{equation}
\mu^{y}\left(B_{\epsilon n^{-\kappa}}^{1}(\truth)\right)\le Z(n,\xi)\exp\left(-\epsilon^{2}n^{1-2\kappa}+n^{\frac{1}{2}-\kappa\lambda}\epsilon^{\lambda}K_{n}\right).\label{eq:postConSmallNotTailUpper}
\end{equation}
We now derive sufficient conditions for $n^{1-2\kappa}$ to be the
dominant term in the exponential in the Equations (\ref{eq:postConSmallNotTailLower})
and (\ref{eq:postConSmallNotTailUpper}) implying Equation (\ref{eq:RatioLimit}).
This is the case if, in addition to Inequality (\ref{eq:postConSmallNoTailRate1}),
\begin{eqnarray*}
1-2\kappa&> & \frac{1}{2}-\kappa\lambda\text{ and}\\
\mbox{\ensuremath{\log\prior\left(B_{\frac{\epsilon n^{-\kappa}}{2}}^{1}(\truth)\right)}}&\gtrsim & -n^{1-2\kappa}
\end{eqnarray*}
hold. The first line is equivalent to Inequality (\ref{eq:postConSmallNoTailRate1})
and using Inequality (\ref{eq:yhmNoTailSmallBall}) the second line
is implied by 
\begin{equation}
1-2\kappa>\kappa\rho.\label{eq:postConSmallNoTailRate2}
\end{equation}
Thus, the Inequalities (\ref{eq:postConSmallNoTailRate1}) and (\ref{eq:postConSmallNoTailRate2})
imply that $-n^{1-2\kappa}$ is the dominant term in the Inequalities
(\ref{eq:postConSmallNotTailLower}) and (\ref{eq:postConSmallNotTailUpper})
establishing Equation (\ref{eq:RatioLimit}). Letting $\gamma\rightarrow0$
concludes the proof.
\end{proof}

\subsubsection{Extension to the Case of Unbounded Priors}

In the following we weaken the assumptions of Theorem \ref{thm:NoTail}
by assuming that the prior has exponential moments of $\bigl\Vert\cdot\bigr\Vert_{s}^{e}$. The price we pay is that the algebraic rate of convergence is implicitly
given as a low-dimensional optimisation problem.

\begin{thm}
\label{thm:postConSmallGeneral}Suppose that the noise satisfies Assumption
\ref{ass:noiseTraceSumm}, the prior satisfies the small ball asymptotic
\[
\log(\prior(B_{\epsilon}^{1}(\truth))\succsim-\epsilon^{-\smallBallAsymp}
\]
 and $\int\exp(3f\left\Vert \input\right\Vert _{s}^{e})d\prior(\input)$$<\infty$
for $f>0$ and $e>0$ for $s>1+\sigma_{0}$. If the following optimisation
problem has a solution $\kappa^{\star}>0$, then for any $\kappa<\kappa^{\star}$
the posterior $\posterior$ is consistent in $\mathcal{H}^{1}$ for
$\truth$ in $\mathcal{H}^{s}$ with rate $n^{-\kappa}$. 

\vspace{0.5cm}

\textup{\noindent$\text{Maximize }\kappa\text{ with respect to }\kappa,p\ge1,\eta,\theta\ge0\text{ subject to}$}
\begin{eqnarray*}
\frac{1}{2}+\eta\frac{p}{q}-\kappa\lambda p&<1-2\kappa \hspace{5cm} &\mbox{(\ref{eq:postConSmallRate1})}\\
\frac{1}{2}-\HilbertScalesYoungs+(1-\HilbertScalesInter)q\nRadiusSnorm&<1-2\deltarate & \mbox{(\ref{eq:postConSmallRate2})}\\
\rho\kappa & <e\theta & \mbox{(\ref{eq:postConSmallRate3})}\\
\rho\kappa&<1-2\kappa &\mbox{(\ref{eq:postConSmallrate4})}\\
\lambda p & <2 &\mbox{(\ref{eq:postConSmallRate5})}\\
\left(\eta\frac{p}{q}-\frac{1}{2}\right)\frac{1}{2-\lambda p} & <-\kappa &\mbox{(\ref{eq:postConSmallRate6})}\\
(1-\HilbertScalesInter)q & <e&\mbox{(\ref{eq:postConSmallRate7})}\\
\left(\frac{1}{2}-\HilbertScalesYoungs\right)\left(1+\frac{1}{e-(1-\HilbertScalesInter)q}\right) & <\max(1-2\kappa,\nRadiusSnorm e)\quad  &\mbox{(\ref{eq:postConSmallRate8})}
\end{eqnarray*}
where $\lambda:=\frac{s-1+\sigma_{0}}{s-1}$. \end{thm}
\begin{proof}
See Appendix \ref{sec:Proof-of-Theorem}.\end{proof}
\begin{rem}
\label{rem:postConGeneralS}In general, $e(s)$ might depend on $s$
for $\int\exp(3f\left\Vert \input\right\Vert _{s}^{e})d\prior(\input)$$<\infty$
to hold. Therefore the rate might be improved by optimising over different
$s>1+\sigma_{0}$. 
\end{rem}
Whereas the algebraic rate in Theorem \ref{thm:postConSmallGeneral}
is implicit, the following corollary yields a simple condition implying
posterior consistency.
\begin{cor}
\label{cor:postConSmallNoRate}Suppose that the noise satisfies Assumption
\ref{ass:noiseTraceSumm}, \textup{$\truth\in\text{supp}_{\mathcal{H}^{1}}(\prior)$}
and $\int\exp(3f\left\Vert \input\right\Vert _{s}^{e})d\prior(\input)$$<\infty$
for $f>0$, $e>0$ and $s>1+\sigma_{0}$. If one of the following
two conditions holds
\begin{eqnarray*}
0<\lambda\le\frac{1}{2} &  \text{ and }\quad &e>-1+2\sqrt{2}\sqrt{1-\lambda}\quad \text{or }\\
\frac{1}{2}<\lambda<1 &   \text{and }\quad & e>2-2\lambda,
\end{eqnarray*}
then $\posterior$ is posterior consistent for $\truth$ in $\mathcal{H}^{s}$.\end{cor}
\begin{proof}
It follows from the proof of Theorem \ref{thm:postConSmallGeneral}
that we only have to find $\eta$, $\theta\ge0$, $p\ge1$ and $s$
such that the Inequalities (\ref{eq:postConSmallRate1}), (\ref{eq:postConSmallRate2}),
(\ref{eq:postConSmallRate5}), (\ref{eq:postConSmallRate7}) and (\ref{eq:postConSmallRate8})
are satisfied. Choosing $\eta$ as large as Inequality (\ref{eq:postConSmallRate1})
permits, that is $\eta:=\frac{1}{2(p-1)}-\epsilon$, extends the range
of solutions of the other inequalities ((\ref{eq:postConSmallRate2})
and (\ref{eq:postConSmallRate8})) containing $\eta$. Similarly,
choosing $\theta$ as large as (\ref{eq:postConSmallRate2}) permits,
that is $\theta:=\frac{0.5+\eta}{(1-\lambda)q}-\epsilon$, extends
the range of solutions of Inequality (\ref{eq:postConSmallRate8}).
Letting $\epsilon\rightarrow0$ in (\ref{eq:postConSmallRate8}) yields
\begin{eqnarray}
p & \ge 1\nonumber \\
\lambda p & <2\hspace{10cm} \mbox{(\ref{eq:postConSmallRate5})}\nonumber \\
(1-\HilbertScalesInter)q & <e\hspace{10cm} \mbox{(\ref{eq:postConSmallRate7})}\nonumber
\end{eqnarray}\begin{eqnarray}
\frac{(p-2)\left(\frac{p-1}{e(p-1)+(\lambda-1)p}+1\right)}{2(p-1)} & <\max\left(1,\frac{e}{2-2\lambda}\right).\label{eq:postConNoRate1}
\end{eqnarray}

Now it is left to perform a case-by-case analysis. Starting from Inequality
(\ref{eq:postConNoRate1}), the first two cases are $\frac{e}{2-\lambda}<1$
and $\frac{e}{2-\lambda}\ge1$. For these cases we have to treat $e(-1+p)+p(-1+\lambda)<0$
and $e(-1+p)+p(-1+\lambda)\ge0$ separately in order to rearrange
Equation (\ref{eq:postConNoRate1}) to a quadratic inequality in $p$.
The details are tedious but straightforward algebra.\end{proof}
\begin{rem}
We would like to point out that the Remarks \ref{Rem:NoRateCorollary-Gaussian}
and \ref{Rem:NoRateCorollary-LogConcave} are also valid for this
more general Corollary \ref{cor:postConSmallNoRate}.
\end{rem}

\subsubsection{Comparison for the Special Case of Gaussian Priors}

In the special case of jointly diagonalisable prior and noise covariance,
we evaluate the consistency rate in Theorem \ref{thm:postConSmallGeneral}
by comparing it with the optimal rates obtained in \cite{knapik2011bayesian}.
By numerically solving the optimisation problem in Theorem \ref{thm:postConSmallGeneral},
we indicate that our rates are close to the optimal rate. 

In the following we first derive a Gaussian prior and noise for a
regression problem before reformulating our result in this context.
In a second step we reformulate the problem in the notation of \cite{knapik2011bayesian}
and state the corresponding result. We conclude this section by an
actual comparison between the posterior consistency rate obtained
in \cite{knapik2011bayesian} and the results of this paper.

\medskip{}

We suppose that the prior is Gaussian $\prior=\mathcal{N}(0,\mathcal{C}_{0})$
and that the covariance operators $\mathcal{C}_{0}$ of the prior
and $\Gamma$ of the noise are jointly diagonalisable over $\{e_{i}\}$
denoting an orthonormal basis of eigenvectors. Furthermore, we assume
that the eigenvalues $\mu_{j}^{2}$ and $\lambda_{j}^{2}$ of $\mathcal{C}_{0}$
and $\Gamma$ satisfy
\begin{eqnarray}
\mu_{j} & =j^{-t}\label{eq:postConSmallGaussPrior}\\
\lambda_{j} & =j^{-r},\label{eq:postConSmallGaussNoise}
\end{eqnarray}
respectively. The inner product of the Hilbert scale with respect
to $\Gamma$ can now explicitly be written as
\[
\left\langle x,y\right\rangle _{r}=\sum_{j=1}^{\infty}\mu_{j}^{-2r}x_{j}y_{j},\text{ }\left\Vert x\right\Vert _{r}^{2}=\sum_{j=1}^{\infty}\mu_{j}^{-2r}x_{j}^{2}.
\]
Moreover, we remark that Assumption \ref{ass:noiseTraceSumm} is satisfied
with $\sigma_{0}=\frac{1}{2r}$. The covariance operator $\tilde{\mathcal{C}_{0}}$
of $\prior$ on $\mathcal{H}^{s}$ has eigenvalues $\mu_{j}\vert_{\mathcal{H}^{s}}=j^{-t+rs}$
which can be seen by denoting $S^{a}e_{k}:=k^{a}e_{k}$ and calculating
\begin{eqnarray}
\mathbb{E}_{\prior}\left\langle x,u\right\rangle _{\mathcal{H}^{s}}\left\langle x,v\right\rangle _{\mathcal{H}^{s}} & =\mathbb{E}_{\prior}\left\langle x,S^{2sr}u\right\rangle _{\mathcal{H}}\left\langle x,S^{2sr}v\right\rangle _{\mathcal{H}}\label{eq:determineCov} \\[5pt]
 & =\left\langle \priorCov S^{2sr}u,S^{2sr}u\right\rangle _{\mathcal{H}}=\left\langle S^{2sr}\priorCov u,v\right\rangle _{\mathcal{H}^{s}}.\nonumber 
\end{eqnarray}
In order to conclude that $\tilde{\mathcal{C}_{0}}$ is trace-class
on $\mathcal{H}_{s}$, we need to impose that $t>rs+\frac{1}{2}$
. In this case, we know from Example 2 and Proposition 3 in Section
18 of \cite{MR1472736} that the small balls asymptotic 
\[
\log(\prior(B_{\epsilon}^{1}(\truth))\succsim-\epsilon^{-\smallBallAsymp}
\]
is satisfied for $\prior$ with $\rho=\frac{-1}{t-r-1}$. 

For this problem we adapt Theorem \ref{thm:postConSmallGeneral} by
optimising over $s$ in the appropriate range as described in Remark
\ref{rem:postConGeneralS}. Moreover, Fernique's theorem \cite{gaussianMeasureas}
for Gaussian measures motivates us setting $e=2$ and $\rho=\frac{-1}{t-r-1}$
as discussed above.
\begin{cor}
\label{cor:gaussianCase} Let the prior and the observational noise
be specified as in Equation (\ref{eq:postConSmallGaussPrior}) and
(\ref{eq:postConSmallGaussNoise}). If the following optimisation
problem has a solution $\kappa^{\star}>0$, then for any $\kappa<\kappa^{\star}$
the posterior $\posterior$ is consistent in $\mathcal{H}^{1}$ for
$\truth$ in $\mathcal{H}^{s}$ with rate $n^{-\kappa}$. 

\noindent$\text{Maximize }\kappa\text{ with respect to }\kappa,\: p\ge1,\eta,\theta\ge0,1+\frac{1}{2r}<s<\frac{t-\frac{1}{2}}{r}\text{ subject to}$
\begin{eqnarray}
\frac{1}{2}+\eta\frac{p}{q}-\kappa\lambda p\nonumber &<&1-2\kappa \\
 \frac{1}{2}-\HilbertScalesYoungs+(1-\HilbertScalesInter)q\nRadiusSnorm&<&1-2\deltarate \nonumber \\
\frac{1}{t-r-1}\kappa & < & 2\theta\nonumber \\
\frac{1}{t-r-1}\kappa&<&1-2\kappa \label{eq:postConSmallNoiseSmallBall}\\
\left(\eta\frac{p}{q}-\frac{1}{2}\right)\lambda p & < & -\kappa\nonumber \\
\lambda p & < & 2\nonumber \\
(1-\HilbertScalesInter)q & < & 2\nonumber \\
\left(\frac{1}{2}-\HilbertScalesYoungs\right)\left(1+\frac{1}{2-(1-\HilbertScalesInter)q}\right) & < & \max(1-2\kappa,\nRadiusSnorm2)\nonumber 
\end{eqnarray}
where $\lambda:=\frac{s-1-\sigma_{0}}{s-1}$. 
\end{cor}
We now recast our problem reformulating it in the setting and notation
of \cite{knapik2011bayesian}. Letting $\zeta$ be $H$-valued white
noise, our problem corresponds to recovering $\input$ from 
\[
y=\input+\frac{1}{\sqrt{n}}\obsCov^{\frac{1}{2}}\zeta.
\]
This problem is equivalent to

\begin{equation}
\tilde{Y}=K\input+\frac{1}{\sqrt{n}}\zeta\label{eq:vaartInverseProblem}
\end{equation}
where $K=\Gamma^{\frac{1}{2}}.$ Let $\left\{ f_{n}\right\} $ be
an orthonormal basis of eigenvectors of $\obsCov$ on $H$. In order
to adapt the notation of \cite{knapik2011bayesian}, we write $H_{2}:=H$
and note that $H_{1}$ will be equivalent to the Cameron-Martin space
which takes the form 
\[
H_{1}=S_{\mathcal{H}_{2}}^{r}:=\left\{ v\in\mathcal{H}_{2}\vert v=\sum v_{i}f_{i}\text{ s.t. }\sum v_{i}^{2}i^{2r}<\infty\right\} 
\]
with orthonormal basis $e_{k}=f_{k}/k^{r}$. Moreover, let $K:H_{1}\rightarrow H_{2}$
be defined as

\[
Ke_{k}:=\obsCov^{-\frac{1}{2}}e_{k}=\frac{\lambda_{k}}{k^{r}}f_{k}.
\]
In order to match Assumption 3.1 in \cite{knapik2011bayesian}, we
have to bound the eigenvalues $\kappa_{i}$ of $K^{T}K$ as follows
\[
M^{-1}i^{-p}\leq\kappa_{i}\leq Mi^{-p}.
\]
We determine these eigenvalues by noting that 
\[
\left\langle K^{T}f_{k},e_{j}\right\rangle _{H_{2}}=\left\langle f_{k},Ke_{j}\right\rangle _{H_{2}}=\delta_{jk}\frac{\lambda_{k}}{k^{r}}.
\]
The calculation above yields 
\[
K^{T}Kf_{k}=\left(\frac{\lambda_{k}}{k^{r}}\right)^{2}f_{k}
\]
and thus
\[
\kappa_{k}=\left(\frac{\lambda_{k}}{k^{r}}\right)^{2}\asymp1=n^{0}\Rightarrow p=0.
\]
As in Equation (\ref{eq:determineCov}), we identify the covariance
operator of $\prior$ on $H_{1}$ through its eigenvalues 
\[
\tilde{\lambda}_{k}\asymp k^{-2t+2r}.
\]
By Theorem 4.1 in \cite{knapik2011bayesian} the posterior contraction
rate is given by
\[
n^{-\frac{\alpha\wedge\beta}{1+2\alpha+2p}}
\]
where $-1-2\alpha=-2t+2r$ (compare Equation (3.5) in \cite{knapik2011bayesian})
and $\beta$ is the regularity of the truth. As above, we suppose that $\beta\ge\alpha$ resulting in 
\[
\kappa_{\text{opt}}=\frac{t-r-\frac{1}{2}}{2(t-r)-1}.
\]
In Figure \ref{fig:comparison}, we use numerical optimisation to
compare our rate to the optimal one for $r=1$ with varying $t$. 

Just considering Inequality (\ref{eq:postConSmallNoiseSmallBall})
(essential to our approach since this implies that the Cameron-Martin
term dominates the prior measure c.f. Equation (\ref{eq:postConHilbertScaleDeltaBall}))
yields

\[
\kappa_{\text{Possible}=\frac{t-r-1}{2(t-r)-1}}
\]
which coincides with the rate $\kappa_{\text{Cor}}$ obtained by solving
the optimisation problem in Corollary \ref{cor:gaussianCase}. Thus,
even if we are able to improve our bounds, there is a genuine gap
between our rate and the optimal rate in the case of Gaussian priors.
The reason for this gap is that Theorem \ref{thm:postConSmallGeneral}
is applicable to any prior satisfying the stated regularity and small
ball assumptions. Nevertheless, Figure \ref{fig:comparison} indicates
that the obtained rates are quite close. In contrast, \cite{knapik2011bayesian}
is only applicable to Gaussian priors for which the Gaussian stucture
of the prior and the posterior are explicitly used.

\begin{figure}
\begin{centering}
\includegraphics[width=0.79\textwidth]{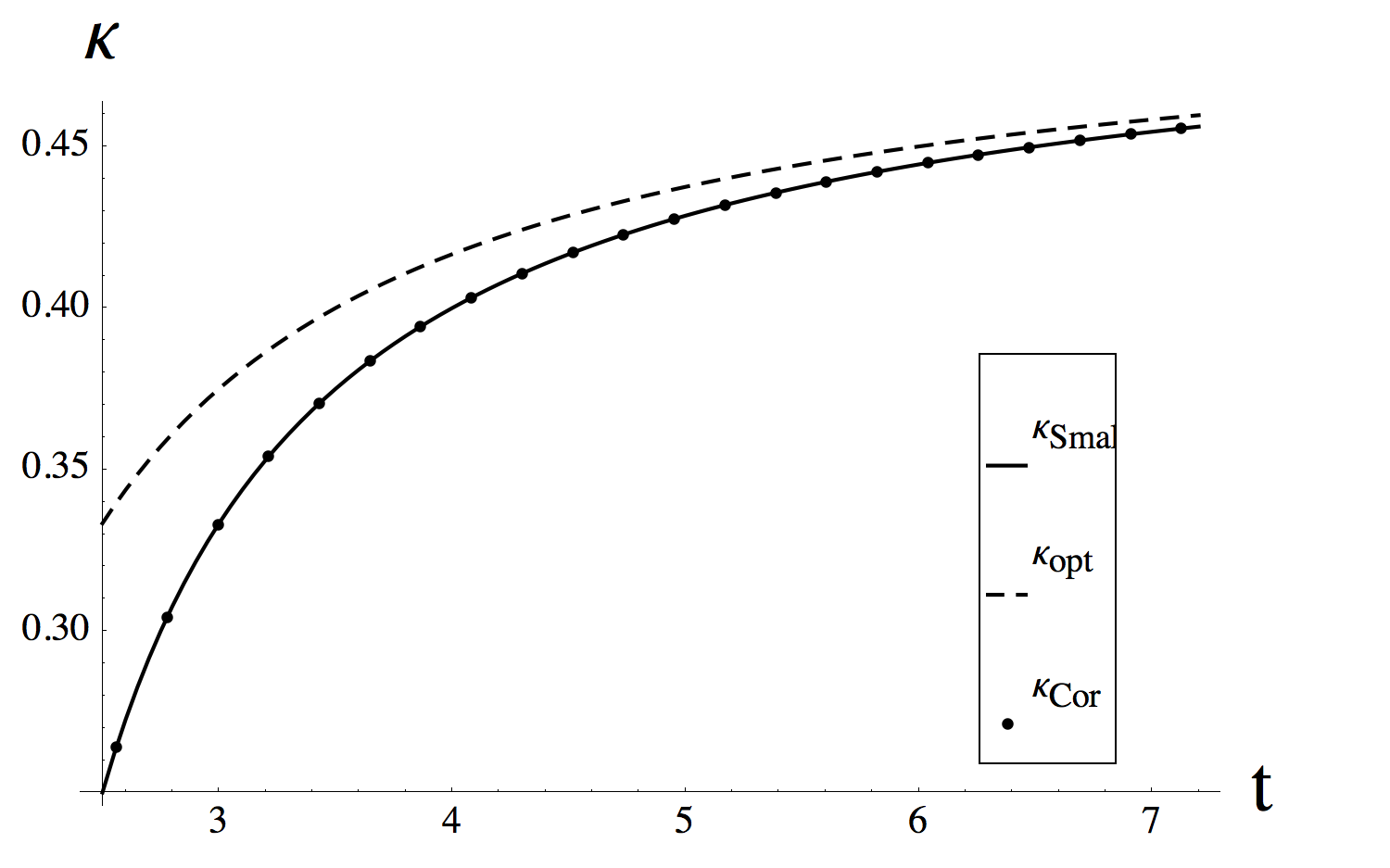}
\par\end{centering}
\caption{\label{fig:comparison}Posterior consistency rate for the Bayesian
regression model with the noise and prior given in Equations (\ref{eq:postConSmallGaussPrior})
and (\ref{eq:postConSmallGaussNoise}). We denote the rate obtained
in \cite{knapik2011bayesian} and the one based on Corollary \ref{cor:gaussianCase}
as $\kappa_{\text{Opt}}$ and $\kappa_{\mbox{Cor}}$,
respectively. We also plot $\kappa_{\mbox{Small Ball}=\frac{t-r-1}{2(t-r)-1}}$
an upper bound on the rate that is obtainable with our method which
is based on the small ball asymptotics of the prior.}
\end{figure}

\subsection{Pointwise Observations in the Large Data Limit\label{sub:Point-wise-Observations-in}}

We consider the following non-parametric Bayesian regression problem

\begin{equation}
\data_{i}=\input(x_{i})+\noise_{i}\quad i:=1,\dots,n\label{eq:largeDataModel}
\end{equation}
with $\input:D\rightarrow\mathbb{R}$, $D$ a bounded domain and $\xi_{i}\overset{\text{i.i.d.}}{\sim}\mathcal{N}(0,\sigma^{2})$.
We assume that a prior $\prior$ is supported on $C(D,\mathbb{R})$
resulting in a posterior of the form
\[
\frac{d\posterior}{d\prior}\propto\exp\left(-\sum_{i=1}^{n}\frac{(\input(x_{i})-y_{i})^{2}}{2\sigma^{2}}\right).
\]
Subsequently, we will prove posterior consistency for this problem
for the case $D=[0,1]$. However, the same reasoning applies to any
bounded domain $D\subseteq\mathbb{R}^{d}$ but the actual posterior
consistency rate depends on $d$. 

As in the previous section, we suppose that the data $\data_{i}$
in Equation (\ref{eq:largeDataModel}) is generated for a fixed 'truth'
$\truth$. Hence
\begin{eqnarray}
y_{i} & = & \truth(x_{i})+\noise_{i}\nonumber \\
\frac{d\mu^{y_{1:n}}}{d\prior} & \propto & \exp\left(-\sum_{i=1}^{n}\frac{\left(\input(x_{i})-\truth(x_{i})\right){}^{2}+2\left(\input(x_{i})-\truth(x_{i})\right)\noise_{i}}{2\sigma^{2}}\right).\label{eq:postConLargeData}
\end{eqnarray}
In this setup posterior consistency depends on the properties of the
prior as well as on the sequence $\left\{ x_{i}\right\} _{i\in\mathbb{N}}$.
In the following, we discuss appropriate assumptions on both giving
rise to Theorem \ref{thm:largeData}. Moreover, we relate this result
with its assumptions to the literature. 
\begin{assumption}
\label{ass:postConDataPriorAss}There exist $\beta\in(0,1]$ and $L>0$
such that
\[
\left\Vert \input\right\Vert _{\beta}\leq L\text{ and }\bigl\Vert\input\bigr\Vert_{\infty}\leq L\quad\prior\text{-a.s.}.
\]

\end{assumption}
As $n$ increases, we gain more and more information about the function
$\input$. In particular, if $\left\{ x_{i}\right\} _{i\in\mathbb{N}}$
is dense in $[0,1]$ it is even possible to reconstruct the value
of $\truth(x)$ from $y_{i}$. More precisely, let $x\in[0,1]$ be
arbitrary, then there are $\left|x_{n_{j}}-x\right|\leq\frac{1}{j^{\frac{2}{\beta}}}$
such that 
\[
\input(x)=\lim_{J\rightarrow\infty}\frac{1}{J}\sum_{j=1}^{J}\input(x_{n_{j}}).
\]

However, we will see that this is not sufficient for posterior consistency.
In fact, we will give an example of posterior inconsistency for this
case. So far, the problem of posterior consistency for this type of
regression problems has mainly been investigated for random evaluation
points $x_{i}$ which are known as random covariates. Appropriate
results of this type can be found in\textbf{ \cite{shen2001rates,walker2004new}}.
An exception is \cite{choi2007posterior} where posterior consistency
without a rate with respect to the $L^{1}$-norm for deterministic
$x_{i}$ is shown. This result is obtained under the following assumption.
\begin{assumption}
\cite{choi2007posterior}\label{ass:choi} Suppose that there exists
a constant $K$ such that whenever $b-a\ge\frac{1}{Kn}$ for $0<a<b<1$
there is at least one $i$ such that $x_{i}\in(a,b)$. 
\end{assumption}
The above condition guarantees that the number of observations in
each interval satisfies a lower bound proportional to its size. More
precisely, an interval of size $n^{-\kappa}$ has at least order $n^{1-\kappa}$
for $n$ being large enough. This can be seen by chopping the interval
into intervals of size $\frac{1}{Kn}$ . In contrast to \cite{choi2007posterior},
we are also able to obtain a posterior consistency rate under this
assumption. Posterior consistency without a rate can be concluded
under the following weaker Assumption.
\begin{assumption}
\label{ass:equally spaced} We suppose that for $\left\{ x_{i}\right\} _{i\in\mathbb{N}}$
there exists a $K>0$ such that for any $a<b\in[0,1]$ there is an
$N(a,b)$ such that
\[
F_{n}(b)-F_{n}(a)\ge K(b-a)\quad\forall n>N(a,b)
\]
where $F_{n}$ denotes the empirical distribution of $\{x_{i}\}_{i=1}^{n}.$ \end{assumption}
\begin{thm}
\label{thm:largeData}Suppose that the Assumptions \ref{ass:postConDataPriorAss}
and \ref{ass:equally spaced} are satisfied. Then $\mu^{y_{1:n}}$
is posterior consistent with respect to the $L^{\infty}$-norm for
any $a^{\dagger}\in\text{supp}_{C^{\beta}}(\prior)$. Moreover, if
Assumption \ref{ass:choi} and the small ball asymptotic 
\[
\log(\prior(B_{\epsilon}^{L^{\infty}}(\truth))\succsim-\epsilon^{-\smallBallAsymp}
\]
are satisfied, then $\mu^{y_{1:n}}$ is posterior consistent with
respect to the $L^{\infty}$-norm with any rate $n^{-\kappa}$ and
\[
\kappa<\min\left\{ \frac{1}{2(2+\frac{1}{\beta})},\frac{2\beta}{(2\beta+1)(2+\rho)}\right\} .
\]

\end{thm}
\begin{proof}

As in Theorem \ref{thm:NoTail}, posterior consistency is implied
by
\[
\sup_{(\xi_{1},\dots,\xi_{n})\in S_{n}}\frac{\mu^{y_{1:n}}(B_{\epsilon n^{-\kappa}}^{L^{\infty}}(\truth)^{c})}{\mu^{y_{1:n}}\left(B_{\epsilon n^{-\kappa}}^{L^{\infty}}(\truth)\right)}\rightarrow0\text{ for }n\rightarrow\infty
\]
for increasing sets $S_{n}$ such that $\noiseMeasure$$\left((\noise_{1},\dots,\noise_{n})\in S_{n}\right)\rightarrow1$.
For notational convenience we write $h:=\input-\truth,$ $S:=\sqrt{\sum_{i=1}^{n}h(x_{i})^{2}}$
and we denote by $\eta$ a generic $\mathcal{N}(0,\sigma^{2})$ random
variable. This allows us to rewrite the posterior in Equation (\ref{eq:postConLargeData})
as 
\begin{equation}
\frac{d\mu^{y_{1:n}}}{d\prior}\propto Z(n,\eta)\exp\left(-\frac{1}{2\sigma^{2}}(S^{2}+2S\eta)\right).\label{eq:thmPostConLargeData}
\end{equation}
Since $y_{1:n}$ is finite dimensional, it is easy to see that $Z(n,\eta)$
is bounded from above and below. Again, fixing $\gamma>0$, we only
need to consider $\eta\in B_{n^{\gamma}}(0)$. Thus, for $0<l_{\epsilon}<1$
we have a lower bound on 
\[
\mu^{y_{1:n}}\left(B_{\epsilon}^{L^{\infty}}\left(\truth\right)\right)\ge Z(n,\eta)\exp\left(-n\frac{l_{\epsilon}^{2}\epsilon^{2}}{2\sigma^{2}}-\frac{l_{\epsilon}\epsilon n{}^{\frac{1}{2}}n^{\gamma}}{\sigma^{2}}\right)\prior\left(B_{l_{\epsilon}\epsilon}^{L^{\infty}}\left(\truth\right)\right).
\]
In order to derive an upper bound on $\mu^{y_{1:n}}(B_{\epsilon}^{L^{\infty}}(\truth)^{c})$,
let $\input\in B_{\epsilon}^{L^{\infty}}(\truth)^{c}$ be chosen arbitrarily
and notice that $f(S)=-S^{2}+Sn^{\gamma}$ is decreasing for $S>n^{\gamma}$.
The upper bound on $\mu^{y_{1:n}}(B_{\epsilon}^{L^{\infty}}(\truth)^{c})$
therefore boils down to a lower bound on $S$ that is larger than
$n^{\gamma}.$ In fact, there is $\hat{x}$ such that $\left|\truth(\hat{x})-\input(\hat{x})\right|\ge\epsilon$.
Applying H\"older continuity yields 
\[
\left|\truth(x)-\input(x)\right|\ge\epsilon/2\quad\text{for }x\in\left(\hat{x}-\Delta x,\hat{x}+\Delta x\right]
\]
for $\Delta x=\left(\frac{\epsilon}{4L}\right)^{\frac{1}{\beta}}$.
Let $I$ be the following index set 
\[
I=\left\{ i\vert x_{i}\in\left(\hat{x}-\Delta x,\hat{x}+\Delta x\right]\right\} .
\]
For $n$ larger than $N_{\epsilon}=\max\left\{ N(i\Delta x,(i+1)\Delta x)\vert i=0\dots\bigl\lfloor1/\Delta x\bigr\rfloor-1\right\} $
it follows that 
\[
K\frac{1}{2}\Delta x\le F_{n}\left(\hat{x}+\Delta x\right)-F_{n}\left(\hat{x}-\Delta x\right)=\frac{\left|I\right|}{n}.
\]
If we only consider $x_{i}$ with $i\in I$, we obtain that
\[
S\ge\sqrt{\frac{\epsilon^{2}}{4}nK\frac{\Delta x}{2}}
\]
which gives rise to the following upper bound

\begin{equation}
\mu^{y_{1:n}}\left(B_{\epsilon}^{L^{\infty}}\left(\truth\right){}^{c}\right)\leq Z(n,\noise)\exp\left[-\frac{n\epsilon^{2}K\Delta x}{8\sigma^{2}}+\left(\frac{K}{2}\Delta x\right){}^{\frac{1}{2}}\frac{\epsilon}{4\sigma^{2}}n^{\frac{1}{2}+\gamma}\right].\label{eq:thmLargeDataPostConUpper}
\end{equation}
By choosing $l_{\epsilon}$ small enough, we also know that 
\[
\frac{\mu^{y_{1:n}}\left(B_{\epsilon}^{L^{\infty}}(\truth)^{c}\right)}{\mu^{y_{1:n}}\left(B_{\epsilon}^{L^{\infty}}(\truth)\right)}\rightarrow0\quad\text{as }n\rightarrow\infty.
\]
In order to obtain a rate of posterior consistency, we use $\tilde{\kappa}>\kappa$
and hence 
\begin{eqnarray}
\hspace{-2.5cm}\posterior\left(B_{n^{-\kappa}}^{L^{\infty}}\left(\truth\right)\right)\ge\posterior\left(B_{n^{-\tilde{\kappa}}}^{L^{\infty}}\left(\truth\right)\right)\ge Z(n,\noise)\exp\left[-\frac{1}{2\sigma^{2}}n^{1-2\tilde{\kappa}}-\frac{1}{2\sigma^{2}}n^{\frac{1-\tilde{\kappa}}{2}+\gamma}-cn^{\tilde{\kappa}\rho}\right].\label{eq:postConLargeLower}
\end{eqnarray}
Thus, Equation (\ref{eq:thmLargeDataPostConUpper}) implies that 
\begin{equation}
\hspace{-1.5cm}\mu^{y_{1:n}}\left(B_{n^{-\kappa}}^{L^{\infty}}\left(\truth\right){}^{c}\right)\leq Z(n,\noise)\exp\left(-\frac{K}{8\sigma^{2}(4L)^{\frac{1}{\beta}}}n^{1-(2+\frac{1}{\beta})\kappa}+\frac{K^{\frac{1}{2}}n^{\frac{1}{2}-\kappa(1+\frac{1}{2\beta})}}{2\sigma^{2}\sqrt{2}(4L)^{\frac{1}{\beta}}}\right).\label{eq:postConLargeUpper}
\end{equation}
The first term in the exponential in Equation (\ref{eq:postConLargeLower})
is dominant over the corresponding term in Equation (\ref{eq:postConLargeUpper})
by choosing 
\[
\tilde{\kappa}:=\kappa(1+\frac{1}{2\beta})+\gamma.
\]
Moreover, the first\textbf{ }term in the Equations (\ref{eq:postConLargeLower})
and (\ref{eq:postConLargeUpper}) is dominant over the other terms
respectively if 
\begin{eqnarray*}
1-2\tilde{\kappa} & > & \tilde{\kappa}\rho\\
1-2\tilde{\kappa} & > & \frac{1-\tilde{\kappa}}{2}+\gamma\\
1-\left(2+\frac{1}{\beta}\right)\kappa & > & \frac{1}{2}-\kappa\left(1+\frac{1}{2\beta}\right).
\end{eqnarray*}
These three inequalities are respectively implied by
\begin{eqnarray*}
\frac{1-\rho\gamma-2\gamma}{\left(1+\frac{1}{2\beta}\right)\left(2+\rho\right)} & > & \kappa\\
\frac{2-10\gamma}{6\left(1+\frac{1}{2\beta}\right)} & > & \kappa\\
\text{}\frac{1}{2+\frac{1}{\beta}} & > & \kappa.
\end{eqnarray*}
Choosing $\gamma$ small enough, we see that $\posterior$ is consistent
in $L^{\infty}$ with any rate
\[
\kappa<\min\left\{ \frac{1}{3(1+\frac{1}{2\beta})},\frac{1}{\left(1+\frac{1}{2\beta}\right)(2+\rho)}\right\} .
\]
\end{proof} 
The assumptions in the theorem above can be justified because a slight
violation leads to the example of posterior inconsistency in the next
section.

\subsubsection{Example of Posterior Inconsistency}\label{sub:InconsistencyExample}
In this section, we construct a counterexample to illustrate that
despite the strong Assumption \ref{ass:postConDataPriorAss} it is
not sufficient for $\{x_{i}\}_{i=1}^{n}$ to be dense in order to
establish posterior consistency. Given such a sequence it is always
possible to extract a subsequence satisfying Assumption \ref{ass:equally spaced}.
Even though all the other observations can be viewed as additional,
we will choose a prior so that the posterior sequence is not consistent. 

In the following, we choose the prior concentrated on functions $g$
which are continuous, satisfy $g(\frac{1}{2})=0$ and are linear on
$[0,\frac{1}{2}]$ and\textbf{ }$[\frac{1}{2},1]$. By identifying
$g(0)$ and $g(1)$ with the first and second component respectively
the following two-dimensional example can be extended to the setting
of Equation (\ref{eq:largeDataModel}). This extension is an example
of posterior inconsistency with respect to the $L^{p}$-norm for $1\le p\le\infty$
because any of these norms is equivalent to $\left\Vert (g(0),g(1))\right\Vert $
for an arbitrary norm on $\mathbb{R}^{2}$.
\begin{example}
We consider the following prior on $\mathbb{R}^{2}$ 
\[
\prior=M\sum_{k=1}^{\infty}\delta_{\left(\frac{1}{\sqrt{k}},0\right)}\exp\left(-2k^{2}\right)+\delta_{\left(\frac{1}{2\sqrt{k}},1\right)}\exp\left(-k^{2}\right)
\]
and we choose $\truth=(0,0)$ as 'truth'. The data consists of $n$
and $n^{\theta}$ with $0<\theta<1$ being measurements of the form
\[
y_{i}=\truth_{1}+\xi_{i}\text{ and }\tilde{y}_{i}=\truth_{2}+\tilde{\xi}_{i}\text{ with }\xi_{i},\tilde{\xi}_{i}\overset{\text{i.i.d.}}{\sim}\mathcal{N}(0,1),
\]
respectively. Consequently, the posterior takes the form 
\[
\mu^{(y_{1:n},\tilde{y}_{1:n^{\theta}})}\propto\prior(d\input_{1},d\input_{2})\exp\left(-\frac{1}{2}\sum_{j=1}^{n}\left(\input_{1}-\xi_{j}\right){}^{2}-\frac{1}{2}\sum_{j=1}^{n^{\theta}}\left(\input_{2}-\tilde{\xi}_{j}\right){}^{2}\right).
\]
Here, posterior consistency of $\mu^{(y_{1:n},\tilde{y}_{1:n^{\theta}})}$
is equivalent to the statement that for any $K$ there is $l_{n}\uparrow1$
such that 
\[
\mathbb{P}_{\eta}\left(\mu^{(y_{1:n},\tilde{y}_{1:n^{\theta}})}\left(\bigcup_{k=K}^{\infty}(\frac{1}{\sqrt{k}},0)\right)\ge l_{n}\right)\rightarrow1\text{ for }n\rightarrow\infty.
\]
We will not only show that $\mu^{(y_{1:n},\tilde{y}_{1:n^{\theta}})}$
is posterior inconsistent but also that there is $l_{n}\downarrow0$
such that for $A=\bigcup_{k=1}^{\infty}(\frac{1}{k},0)$
\[
\mathbb{P}_{\eta}\left(\mu^{(y_{1:n},\tilde{y}_{1:n^{\theta}})}\left(A\right)\le l_{n}\right)\rightarrow1\text{ for }n\rightarrow\infty.
\]
Because of $\mu^{(y_{1:n},\tilde{y}_{1:n^{\theta}})}(A)+\mu^{(y_{1:n},\tilde{y}_{1:n^{\theta}})}(A^{c})=1$,
we may proceed as in the proofs of the Theorems \ref{thm:NoTail}
and \ref{thm:largeData} and thus it is enough to construct sets of
increasing $\mathbb{P}_{\xi}$-probability such that on these sets
\begin{eqnarray*}
\hspace{-2cm}\frac{\mu^{(y_{1:n},\tilde{y}_{1:n^{\theta}})}(A)}{\mu^{(y_{1:n},\tilde{y}_{1:n^{\theta}})}(A^{c})} & = & \frac{\sum_{k}\exp\left(-\frac{1}{2}\sum_{j=1}^{n}\left(\frac{1}{\sqrt{k}}-\xi_{j}\right){}^{2}-\frac{1}{2}\sum_{j=1}^{n^{\theta}}\tilde{\xi}{}^{2}-2k^{2}\right)}{\sum_{k}\exp\left(-\frac{1}{2}\sum_{j=1}^{n}\left(\frac{1}{2\sqrt{k}}-\xi_{j}\right){}^{2}-\frac{1}{2}\sum_{j=1}^{n^{\theta}}\left(\tilde{\xi}_{j}-1\right){}^{2}-k^{2}\right)}\\[5pt]
 & = & \frac{\sum_{k}\exp\left(-\frac{1}{2}\sum_{j=1}^{n}\frac{1}{k}-\frac{2}{k}\xi_{j}-2k^{2}\right)}{\sum_{k}\exp\left(-\frac{1}{2}\sum_{j=1}^{n}\frac{1}{4k}-\frac{1}{\sqrt{k}}\xi_{j}-\frac{1}{2}\sum_{j=1}^{n^{\theta}}1-2\tilde{\xi}_{k}-k^{2}\right)}\rightarrow0.
\end{eqnarray*}
The $\mathbb{P}_{\xi}$-probabilities of $\vert\sum_{j=1}^{n}\xi_{j}\vert>Mn^{\frac{1}{2}+\gamma}$
and $\vert\sum_{j=1}^{n^{\theta}}\tilde{\xi}_{j}\vert>Mn^{\frac{1}{2}+\gamma}$
are exponentially small in $n$. Thus, it is enough to consider
\begin{eqnarray*}
\hspace{-2cm}\frac{\sum_{k}\exp\left(-\frac{1}{2}n\frac{1}{k}-2k^{2}\right)}{\sum_{k}\exp\left(-\frac{1}{2}n\frac{1}{4k}-k^{2}\right)} & =\frac{\sum_{k}^{\sqrt{n}}\exp\left(-\frac{1}{2}n\frac{1}{k}-2k^{2}\right)+\sum_{\sqrt{n}}^{\infty}\exp\left(-\frac{1}{2}n\frac{1}{k}-2k^{2}\right)}{\sum_{k}^{\sqrt{n}}\exp\left(-\frac{1}{2}n\frac{1}{4k}-k^{2}\right)+\sum_{\sqrt{n}}^{\infty}\exp\left(-\frac{1}{2}n\frac{1}{4k}-k^{2}\right)}\\[5pt]
 & \leq\max\left\{ \exp\left(-\frac{3}{4}\sqrt{n}\right),\exp(-n)\right\} \rightarrow0\text{ as }n\rightarrow\infty.
\end{eqnarray*}
Hence we have shown that $\posterior$ is not posterior consistent.
\end{example}
This example relies on the prior having strong correlations between
its two components. Therefore it seems an interesting question how
the assumptions on $\prior$ can be strengthened in order to relax
those on $\left\{ x_{i}\right\} .$

\subsection{\label{sub:Convergence-in-Stronger}Convergence in Stronger Norms}

We conclude this section by showing that interpolation inequalities
can be used in order to strengthen the norm in which the posterior
concentrates. In particular we consider the small noise limit as described
in Section \ref{sub:The-Small-Noise}.

Suppose we know that the posterior concentrates around the truth $\truth$
in the Cameron-Martin norm $\bigl\Vert\cdot\bigr\Vert_{1}$. In order
to show consistency in $\left\Vert \cdot\right\Vert _{r}$, we write
\begin{eqnarray*}
\left\{ \left\Vert \input-\truth\right\Vert _{r}>\epsilon\right\}  & \subset & \left\{ \left\Vert \input-\truth\right\Vert _{1}^{\lambda}\left\Vert \input-\truth\right\Vert _{s}^{1-\lambda}>\epsilon\right\} \\
 & \subset & \left\{ \left\Vert \input-\truth\right\Vert _{1}^{\lambda}>\frac{\epsilon}{K}\right\} \cup\left\{ \left\Vert \input-\truth\right\Vert _{s}^{1-\lambda}>K\right\} .
\end{eqnarray*}
The posterior probability of the first set is small due to the posterior
consistency in $\mathcal{H}^{1}$. The posterior probability of the
second set is small due to the tails of the prior and the posterior.
Obtaining estimates of this type can be done similarly to the steps
subsequent to Equation (\ref{eq:postConHilbertScalesPosteriorMoment})
in the proof of Theorem \ref{thm:postConSmallGeneral}. Using this
technique, it is also possible to apply the results of this section
to the \EIP\ in the next section. A similar technique based on interpolation
inequalities between H�lder spaces applies to the large data limit
and is also used for the \EIP.

\section{Posterior Consistency for an Elliptic Inverse Problem\label{sec:Results-for-Elliptic}}

In Section \ref{sub:IntroElliptic}, we introduced the idea of reducing
posterior consistency of the \EIP\  to that of the \BRP. For this
example we demonstrate our method for both the small noise and the
large data limit. We start by giving the proof for the small noise
limit in detail before sketching the same steps for the large data
limit.\textbf{ }We emphasise the case of posterior consistency in
the small noise limit because of its analogy with convergence results
for regularisation methods.

\subsection{Posterior Consistency in the Small Noise Limit}

Using Theorem \ref{thm:NoTail} from Section \ref{sec:General-Results}
to conclude posterior consistency of the \EIP\  (c.f. Section \ref{sub:IntroElliptic})
is not entirely straightforward because we have to lift the posterior
consistency for the \BRP\  to $C^{2}$. Moreover, we have to find
appropriate assumptions on the prior $\prior$ so that the push forward
prior $p_{\star}\prior$ satisfies the assumptions of the Theorem
\ref{thm:NoTail}. Again, a rate of posterior consistency is obtained
if the prior satisfies appropriate small ball asymptotics. In a second
step we verify those for the so-called uniform priors which are based
on a series expansion with uniformly distributed coefficients, for
details see below or consider \cite{Con4,2012arXiv1207.2411H,stuartchinanotes}. 

In order to formulate assumptions on $\prior$ implying that $p_{\star}\prior$
satisfies the assumptions of Theorem \ref{thm:NoTail}, we assume
for simplicity that $\noise\sim\mathcal{N}(0,(-\Delta_{\text{Dirichlet}})^{-r})$
where $\Delta_{\text{Dirichlet}}$ denotes the Laplacian with homogeneous
Dirichlet conditions. In this case the abstract Hilbert scale $\mathcal{H}^{s}$
(c.f. \ref{sec:Notation}) corresponds to the standard Sobolev
space $H^{rs}$. Thus, the almost sure bounds in Theorem \ref{thm:NoTail}
are implied by the appropriate assumptions on the prior and classical
results from \cite{MR2597943,MR1814364}.

Moreover, the choice $\noise\sim\mathcal{N}(0,(-\Delta_{\text{Dirichlet}})^{-r})$
also implies that Assumption \ref{ass:noiseTraceSumm} holds for $\sigma_{0}=\frac{d}{2r}$.
This is due to the fact that the operator $(-\Delta_{\text{Dirichlet}})^{-r}$
has eigenvalues $\lambda_{k}^{2}$ with $\lambda_{k}\asymp k^{-2r/d}$
(see Section \ref{sec:Notation} for notation) where $d$ denotes
the dimension of the domain $D$. These results are called Weyl asymptotics
and further details can be found in \cite{MR1511670} and \cite{MR1670907}.

The following theorem summarises the consequences for the posterior
consistency of the \EIP.
\begin{thm}
\label{thm:ellipticNoRate}Suppose that the noise is given by $\xi\sim\mathcal{N}(0,(-\Delta_{\text{Dirichlet}})^{-r})$
and that the prior $\prior$ satisfies 
\[
a(x)\ge\lambda>0\,\forall x\in D\text{ and }\bigl\Vert a\bigr\Vert_{C^{\alpha}}\leq\Lambda\qquad\text{for }\prior\text{-a.e.\,}a\quad\text{and for }\alpha>1.
\]
If $\alpha>r+\frac{d}{2}-2$, $\beta+1>r$ and $\truth\in\text{supp}_{C^{\beta}}\prior$,
then the \EIP\  is posterior consistent with respect to the $C^{\tilde{\alpha}}$-norm
for any $\tilde{\alpha}<\alpha$. Additionally, if 
\[
\log(\prior(B_{\epsilon}^{C^{\beta}}(\truth))\succsim-\epsilon^{-\smallBallAsymp}
\]
then the \EIP\  is posterior consistent with respect to the $L^{\infty}$-norm
with rate $n^{-\kappa}$ for any $\kappa$ such that 
\[
\kappa<\left(\frac{\alpha}{\alpha+2+\frac{d}{2}-r}\wedge1\right)\left(\frac{1}{2+\rho}\wedge\frac{\alpha}{2\left(\alpha+1+\frac{d}{2r}\right)}\right).
\]

\end{thm}
Before proving Theorem \ref{thm:ellipticNoRate}, we notice that forward
stability results (as Proposition \ref{prop:forwardStability}) can
be used to transfer small ball asymptotics from $\prior$ to $\tilde{\prior}=p_{\star}\prior$.
\begin{lem}
\label{lem:pushforwardSmallBall}If the prior satisfies the small
ball asymptotic 
\[
\log\left(\prior(B_{\epsilon}^{C^{\beta}}(\truth)\right)\succsim-\epsilon^{-\smallBallAsymp},
\]
then 
\[
\log\left(\tilde{\prior}(B_{\epsilon}^{C^{\beta+1}}(p^{\dagger})\right)\succsim-\epsilon^{-\smallBallAsymp}.
\]
\end{lem}
\begin{proof}
Proposition \ref{prop:forwardStability} implies that 
\begin{eqnarray*}
\left\Vert \input-\truth\right\Vert _{C^{\beta}} & \leq & \epsilon\Rightarrow\left\Vert p-\pressure^{\dagger}\right\Vert _{C^{\beta}}\leq M\epsilon\\[5pt]
\text{ and }p\left(B_{\epsilon}^{C^{\beta}}\left(\truth\right)\right) & \subseteq & B_{C\epsilon}^{C^{\beta+1}}\left(\pressure^{\dagger}\right).
\end{eqnarray*}
Hence the statement follows.
\end{proof}

Having established Lemma \ref{lem:pushforwardSmallBall}, we are now
in the position to prove the main theorem of this section.

\begin{proof}[Proof of Theorem \ref{thm:ellipticNoRate}]

Subsequently, $M$ will denote a generic constant in different contexts
that may change form line to line. We will first prove posterior consistency
in $L^{\infty}$ before we use an interpolation inequality to bootstrap
it to $C^{\tilde{\alpha}}$. In order to prove posterior consistency
in the $L^{\infty}$-norm, it is enough to show posterior consistency
of the \BRP \  in the $C^{2}$-norm because 
\begin{equation}
\mu^{y_{n}}\bigl(B_{\epsilon}^{L^{\infty}}(\truth)\bigr)=\tilde{\mu}^{y_{n}}\bigl(p(B_{\epsilon}^{L^{\infty}}(\truth)\bigr)\ge\tilde{\mu}^{y_{n}}\left(B_{\frac{\epsilon}{M}}^{C^{2}}(\pressure^{\dagger})\right)\label{eq:thmNotTailNoRateStability}
\end{equation}
which follows by an application of Proposition \ref{prop:Richter}
and a change of variables (see Theorem \ref{thm:IndpendentOfDiscription}).
Using Theorem 6.19 from \cite{MR1814364}, we may conclude that 
\begin{eqnarray*}
\bigl\Vert p\bigr\Vert_{H^{\alpha+2}}\lesssim\bigl\Vert p\bigr\Vert_{C^{\alpha+2}} & \leq K\;\tilde{\mu}_{0}\text{-a.s..}
\end{eqnarray*}
Since $\alpha+2>r$, $p$ is $\tilde{\mu}_{0}$-a.s. an element of
the Cameron-Martin space of $\mu_{\noise}$ as it corresponds to $H^{r}$.
Posterior consistency of the \BRP\  with respect to the $H^{r}$-norm
is now implied by Theorem \ref{thm:NoTail}. Its conditions are satisfied
because
\[
\left\Vert p\right\Vert _{\mathcal{H}^{s}}\leq M\Lambda\quad\tilde{\mu}_{0}\text{-a.s.}
\]
with 
\[
s=\frac{\alpha+2}{r}>1+\frac{d}{2r}=1+\sigma_{0}.
\]

Furthermore, Proposition \ref{prop:forwardStability} and the fact
that $\truth\in\text{supp}_{C^{\beta}}\prior$ imply that $p(\truth)\in\text{supp}_{H^{r}}\tilde{\mu_{0}}$
. In order to bootstrap to posterior consistency in the $C^{2}$-norm,
we use a generalisation of the Sobolev embedding theorem for Besov
spaces and an interpolation inequality between Besov spaces on domains
(for details consult \cite{triebel2008function}). We first note that
$B_{22}^{\tau}=H^{\tau}$ and $C^{\tau}=B_{\infty\infty}^{\tau}$
for $\tau\notin\mathbb{Z}$. In particular Theorem 4.33 in \cite{triebel2008function}
implies that 
\[
\bigl\Vert g\bigr\Vert_{B_{\infty\infty}^{r-\frac{d}{2}-\gamma}}\leq M\bigl\Vert g\bigr\Vert_{H^{r}}
\]
for $\gamma>0$ being small. If $r>\frac{d}{2}+2$, we can conclude
posterior consistency in the $C^{2}$-norm because 
\begin{equation}
B_{\frac{\epsilon}{M}}^{C^{2}}(\pressure^{\dagger})\supseteq\left\{ p\vert\bigl\Vert p-p^{\dagger}\bigr\Vert_{B_{\infty\infty}^{r-\frac{d}{2}-\gamma}}\leq\frac{\epsilon}{M}\right\} \supseteq\left\{ p\vert\bigl\Vert p-p^{\dagger}\bigr\Vert_{H^{r}}\leq\frac{\epsilon}{M}\right\} \label{eq:ellipticSmallNoiseEnough}
\end{equation}
holds for $\gamma$ small enough. Otherwise, we use the interpolation
inequality between Besov spaces subject of Theorem 4.17 in \cite{triebel2008function}
\begin{equation}
\bigl\Vert g\bigr\Vert_{C^{2+\gamma}}\leq\bigl\Vert g\bigr\Vert_{B_{\infty\infty}^{r-\frac{d}{2}-\gamma}}^{\theta}\bigl\Vert g\bigr\Vert_{B_{\infty\infty}^{\alpha+2}}^{1-\theta}\label{eq:ellipticNoRateNoTail}
\end{equation}
for $\gamma$ small enough and with $\theta=\frac{\alpha}{\alpha+2+\frac{d}{2}-r+\gamma}$.
Similar to Equation (\ref{eq:ellipticSmallNoiseEnough}), it follows
that 
\begin{eqnarray}
B_{\frac{\epsilon}{M}}^{C^{2}}(\pressure^{\dagger}) \supseteq \left\{ p\vert\bigl\Vert p-p^{\dagger}\bigr\Vert_{B_{\infty\infty}^{r-\frac{d}{2}-\gamma}}^{\theta}\leq\frac{\epsilon}{KM}\right\} \nonumber \supseteq \left\{ p\vert\bigl\Vert p-p^{\dagger}\bigr\Vert_{H^{r}}\leq\frac{\epsilon^{\theta^{-1}}}{M}\right\} .\label{eq:ellipticSmallNoiseInterp}
\end{eqnarray}
The Equations (\ref{eq:ellipticSmallNoiseEnough}) and (\ref{eq:ellipticSmallNoiseInterp})
allow us to bootstrap the posterior consistency of $\tilde{\mu}^{y_{n}}$
to $C^{2}.$ Equation (\ref{eq:thmNotTailNoRateStability}) implies
posterior consistency of $\mu^{y_{n}}$ in the $L^{\infty}$-norm.
Similarly, we bootstrap to posterior consistency in $C^{\tilde{\alpha}}$
for $\tilde{\alpha}<\alpha$ using the same interpolation technique
as above.

In order to obtain a rate for posterior consistency, we first note
that $\log(\prior(B_{\epsilon}^{C^{\beta}}(\truth))\succsim-\epsilon^{-\smallBallAsymp}$
implies 
\[
\log(\tilde{\prior}(B_{\epsilon}^{r}(\truth))\succsim-\epsilon^{-\smallBallAsymp}
\]
 due to Lemma \ref{lem:pushforwardSmallBall}. Now Theorem \ref{thm:NoTail}
implies posterior consistency of the sequence of posteriors $\tilde{\mu}^{y_{n}}$
in $H^{r}$ with any rate $\kappa$ such that
\[
\kappa<\frac{1}{2+\rho}\wedge\frac{\alpha+1}{2\left(\alpha+1+\frac{d}{2r}\right)}.
\]
Using the interpolation inequality as above gives rise to posterior
consistency for\textbf{ }$\tilde{\mu}^{y_{n}}$ in $C^{2+\alpha}$
with rate $n^{-\kappa}$ for any $\kappa$ such that
\[
\kappa<\left(\frac{\alpha}{\alpha+2+\frac{d}{2}-r}\wedge1\right)\left(\frac{1}{2+\rho}\wedge\frac{\alpha}{2\left(\alpha+1+\frac{d}{2r}\right)}\right).
\]
As above, this implies the same rate of posterior consistency for
$\posterior$ in $L^{\infty}.$

\end{proof}

\subsubsection{Uniform Prior\label{sub:Posterior-Consistency-Rate}}

\global\long\def\sball{\nu}

In this section, we establish a rate of posterior consistency for
the \EIP\  with the so-called uniform prior introduced in \cite{stuartchinanotes,2012arXiv1207.2411H,Con4}.
This choice of the prior was motivated by the preceding analysis in
the uncertainty quantification literature, see for instance \cite{2010SchwabEllipticUQ,UQellipticPDESchwab}.
It is given by

\begin{equation}
\prior=\mathcal{L}\left(a_{0}(x)+\sum_{i=1}^{\infty}\gamma_{i}z_{i}\psi_{i}(x)\right)\quad z_{i}\overset{\text{i.i.d.}}{\sim}\mathcal{U}[-1,1]\label{eq:uniformPrior}
\end{equation}
where $\mathcal{L}$ denotes the law of a random variable. Moreover,
we suppose that $\left\Vert \psi_{i}(x)\right\Vert _{C^{\beta}}=1$,
$\gamma_{i}>0$ and $S=\sum_{i=1}^{\infty}\gamma_{i}<\infty$ such
that
\[
0<a_{\text{min}}\leq a\le a_{\text{max}}\quad\prior\text{-a.s.}.
\]
In order to obtain a rate for the \EIP\  with this prior, we derive
a small ball asymptotic under an appropriate assumption on the decay
of $\{\gamma_{i}\}$. 
\begin{assumption}
\label{ass:coefficientDecay}There exists $\sball^{\star}\in(0,1)$
such that for all $\mathfrak{\sball>\sball^{\star}}$ 
\[
S_{\mathfrak{\sball}}=\left(\sum_{i=1}^{\infty}\gamma_{i}^{\sball}\right)^{\frac{1}{\mathfrak{\sball}}}<\infty.
\]

\end{assumption}
Since the series in Equation (\ref{eq:uniformPrior}) is absolutely
convergent, we assume without loss of generality that $\gamma_{i}$
is decreasing. This allows us to use the following classical inequality
from approximation theory \cite{devore1998nonlinear} 
\begin{equation}
\left(\sum_{n>N}\gamma_{n}\right)\leq N^{1-\frac{1}{\mathfrak{\sball}}}S_{\sball}.\label{eq:summabilityExp}
\end{equation}

\begin{lem}
\label{thm:smallBallPrior}Suppose that $\prior$ is given as in Equation
(\ref{eq:uniformPrior}), Assumption \ref{ass:coefficientDecay} is
satisfied with $\sball^{\star}$ and
\[
\truth=\sum_{i=1}^{\infty}\gamma_{i}z_{i}^{\dagger}\psi_{i}(x)\quad\text{where }z_{i}^{\dagger}\in[-1,1].
\]
Then for any $\mathfrak{\sball>\sball^{\star}}$
\[
\log\prior\bigl(B_{\epsilon}^{C^{\beta}}(\truth)\bigr)\gtrsim-\epsilon^{-\frac{1}{\frac{1}{\sball}-1}}.
\]
\end{lem}
\begin{proof}
We obtain an asymptotic lower bound on the small ball probability
by choosing an appropriate subset $D_{\epsilon}(\truth)$ of $B_{\epsilon}^{C^{\beta}}(\truth)$.
We denote a generic element of this set by 
\[
a=\sum_{i=1}^{\infty}\gamma_{i}z_{i}\psi_{i}(x).
\]
Choosing $N_{\epsilon}$ such that $\sum_{i=N_{\epsilon}}^{\infty}\gamma_{i}\leq\frac{\epsilon}{2}$,
the corresponding terms contribute at most $\frac{\epsilon}{2}$ to
the difference $\left\Vert \truth-\input\right\Vert $. The subset
$D_{\epsilon}(\truth)$ prescribes intervals for $z_{i}$ $i=1\dots N_{\epsilon}$
such that this contribution is at most $\frac{\epsilon}{2}$, too.
More precisely, let $\sball^{\star}<\tilde{\mathfrak{\sball}}<\sball$,
then Equation (\ref{eq:summabilityExp}) implies
\[
\sum_{n>N_{\epsilon}}\gamma_{n}\le N_{\epsilon}^{1-\frac{1}{\mathfrak{\tilde{\sball}}}}S_{\tilde{\sball}}\leq\frac{\epsilon}{2}
\]
 for $N_{\epsilon}\ge\left(\frac{2S_{\tilde{\sball}}}{\epsilon}\right)^{\frac{1}{\frac{1}{\mathfrak{\tilde{\sball}}}-1}}$.
Let the subset $D_{\epsilon}(\truth)\subseteq B_{\epsilon}^{C^{\beta}}(\truth)$
be given by 
\[
\hspace{-2cm}\left\{ a\big\vert\left(z_{i}^{\dagger}>0\wedge z_{i}^{\dagger}-\frac{\epsilon}{2S}\leq z_{i}\le z_{i}^{\dagger}\right)\vee\left(z_{i}^{\dagger}\leq0\wedge z_{i}^{\dagger}+\frac{\epsilon}{2S}\ge z_{i}\ge z_{i}^{\dagger}\right)\,1\leq i\leq N_{\epsilon}\right\} .
\]
Then 
\begin{eqnarray*}
\prior\left(B_{\epsilon}^{C^{\beta}}(\truth)\right) & \ge & \left(\frac{\epsilon}{2S}\right)^{N_{\epsilon}}\\
\log\prior\left(B_{\epsilon}^{C^{\beta}}(\truth)\right) & \gtrsim & N_{\epsilon}\log\epsilon\gtrsim-\epsilon^{-\frac{1}{\sball-1}}.
\end{eqnarray*}
\end{proof}
Combining Lemma \ref{thm:smallBallPrior} and Theorem \ref{thm:ellipticNoRate}
results in the following theorem which characertises posterior consistency
for this class of priors.
\begin{thm}
Let the prior $\prior$ be defined as in Equation (\ref{eq:uniformPrior})
and let Assumption \ref{ass:coefficientDecay} be satisfied. Additionally,
we assume that $\alpha\ge\beta\ge r+1$, $\alpha>r+\frac{d}{2}-2$
and $\Vert a\Vert_{\alpha}\leq K\:\prior-\text{a.s.}.$ Then the posterior
$\posterior$ is consistent for any 
\[
\,\truth=\sum_{i=1}^{\infty}\gamma_{i}z_{i}^{\dagger}\psi_{i}(x)\quad\text{where }z_{i}^{\dagger}\in[-1,1]
\]
with respect to \textup{the $L^{\infty}$-norm} with rate $\epsilon_{n}=M(\kappa)n^{-\kappa}$
for any $\kappa$ such that 
\[
\kappa<\left(\frac{\alpha}{\alpha+2+\frac{d}{2}-r}\wedge1\right)\left(\frac{1-\sball}{2-\sball}\wedge\frac{\alpha-r+2}{2\alpha+d-2r+4}\right).
\]
\end{thm}

\subsection{Posterior Consistency in the Large Data Limit}

In the following we show that the results for the \BRP\  can be transferred
to posterior consistency results in the large data limit for the \EIP.
We consider only the case $d=1$ with $D=[0,1]$ as the general case
is similar. Furthermore, assuming that the observations are of the
form
\[
y_{i}=p(x_{i};a)+\noise\quad i=1\dots n,
\]
the sequence of posteriors is given by 
\[
\frac{d\posterior}{d\prior}(a)\propto\exp\left(-\sum_{i=1}^{n}\frac{\left(p(a)(x_{i})-y_{i}\right)^{2}}{2\sigma^{2}}\right).
\]

Posterior consistency of the \EIP\  in $L^{\infty}$ can then be
derived on the basis of Theorem \ref{thm:largeData}. 
\begin{thm}
Suppose that the sequence $\{x_{i}\}$ satisfies Assumption \ref{ass:equally spaced},
$\bigl\Vert\input\bigr\Vert_{C^{\gamma}}\leq L$ $\prior$-a.s. with
$\gamma>1$ and $\input\ge\input_{min}$ $\prior$-a.s.. If $\truth\in\text{supp}_{C^{\gamma}}\prior$,
then the \EIP\  is posterior consistent in the large data limit with
respect to $C^{\tilde{\gamma}}$ for any $\tilde{\gamma}<\gamma.$
\begin{proof}
An application of Theorem 6.13 in \cite{MR1814364} yields the existence
of $M(D,\gamma)$ so that for all $\input$ satisfying
\[
\bigl\Vert\input\bigr\Vert_{C^{\gamma}}\leq S\text{ and }\input\ge\input_{min}
\]
 there is a unique solution $p$ such that $\left\Vert \pressure\right\Vert _{C^{2+\gamma}}\leq M$.
Thus, $\tilde{\mu}_{0}=p_{\star}\mu_{0}$ satisfies the assumptions
of Theorem \ref{thm:largeData} implying that $\tilde{\mu}^{\data_{1:n}}$
is posterior consistent in $L^{\infty}.$ Using the interpolation
inequality between $L^{\infty}$ and $C^{2+\gamma}$, we also obtain
consistency in $C^{2}$. As in Theorem \ref{thm:ellipticNoRate},
Proposition \ref{prop:Richter} can be used in order to conclude posterior
consistency of $\mu^{\data_{1:n}}$ in $L^{\infty}$. We can bootstrap
from $L^{\infty}(D)$ to $C^{\tilde{\gamma}}$ by interpolating between
$L^{\infty}$ and $C^{\gamma}$.
\end{proof}
\end{thm}

\section{Concluding Remarks\label{sec:Conclusion}}

In this article, we have established a novel link between stability
results for an inverse problem and posterior consistency for the Bayesian
approach to it. We have explicitly shown this link for an elliptic
inverse problem (c.f. \EIP) but the same method is also applicable
for the general case. An instance is electrical impedance tomography
(Calder�n problem) for which stability results are available \cite{alessandrini1988stable}.
This example would lead to a very slow posterior consistency rate
since its stability results are weak. Essentially, we would have to
redo all the calculations on a log-scale instead of an algebraic scale.

So far, we need exponential moments of the prior for the Bayesian
regression of functional response and for pointwise observations (see
also Section 4.2.2 in \cite{choi2007posterior}). For this reason
it is harder to prove posterior consistency for example for log-Gaussian
priors. Log-Gaussian measures have moments of arbitrary order but
no exponential moments. This is a problem that we would like to pursue
further in the future.

\appendix

\section{Notation and Review of Technical Tools\label{sec:Notation} }

\subsection{Asymptotic Inequalities}

We use the following notation for asymptotic inequalities:

Let $a_{n}$ and $b_{n}$ be sequences in $\mathbb{R}$. We denote
by $\mathbb{R}$ $a_{n}\lesssim b_{n}$ that there are $N\in\mathbb{N}$
and $M\in\mathbb{R}$ such that $a_{n}\leq Mb_{n}\text{ for }n\ge N.$
Moreover, if $a_{n}\lesssim b_{n}\lesssim a_{n}$, we write $a_{n}\asymp b_{n}.$

\subsubsection{Hilbert Scales\label{sub:Hilbert-Scales}}

In order to measure the smoothness of the noise and samples of the
prior, we introduce Hilbert scales following \cite{inverseProblem}.
Let $\obsCov$ be a self-adjoint, positive-definite, trace-class linear
operator with eigensystem $\left(\lambda_{k}^{2},\phi_{k}\right)$.
We know that $\obsCov^{-1}$ is a densely defined, unbounded, symmetric
and positive-definite operator because 
\[
H=\overline{\mathcal{R}(\obsCov)}\oplus\text{Ker}(\obsCov)^{\perp}=\overline{\mathcal{R}(\obsCov)}.
\]
We define the\textit{ Hilbert scale} by $\left((\mathcal{H}^{t},\left\langle \cdot,\cdot\right\rangle _{t})\right){}_{t\in\mathbb{R}}$
with $\mathcal{H}^{t}:=\overline{\mathcal{M}}^{\left\Vert \cdot\right\Vert _{t}}$
for 
\begin{eqnarray*}
\mathcal{M} & := & \bigcap_{n=0}^{\infty}\mathcal{D}(\obsCov^{-n})\\
\left\langle u,v\right\rangle _{t} & := & \left\langle \obsCov^{-\frac{t}{2}}u,\obsCov^{-\frac{t}{2}}v\right\rangle \\
\left\Vert u\right\Vert _{t} & := & \left\Vert \obsCov^{-\frac{t}{2}}u\right\Vert .
\end{eqnarray*}
We will denote balls with respect to the $\left\Vert \cdot\right\Vert _{t}$-norm
by 
\[
B_{R}^{t}(u)=\left\{ x\vert\left\Vert u-x\right\Vert _{t}\leq R\right\} .
\]

Moreover, these collection of norms satisfies an interpolation inequality
\begin{prop}
\label{lem:HScaleInterpolation}(Proposition 8.19 in \cite{inverseProblem})
Let $q<r<s$ then the following interpolation inequality holds
\[
\left\Vert x\right\Vert _{r}\leq\left\Vert x\right\Vert _{q}^{\frac{s-r}{s-q}}\left\Vert x\right\Vert _{s}^{\frac{r-q}{s-q}}.
\]
\end{prop}
\begin{rem*}
Our definition here is slightly different from the literature in order
to match it to the Sobolev spaces for $\obsCov=(-\Delta_{\text{Dirichlet}})^{-1}$
\end{rem*}

\subsection{Gaussian Measures}

In this section, we set out our notation for some standard results
about infinite dimensional Gaussian measures which can be found in
the following textbooks and lecture notes \cite{gaussianMeasureas,StochEqnInf,HairerSPDE}.
Let $\gamma$ be a Gaussian measure on a Hilbert space $(H,\left\langle \cdot,\cdot\right\rangle )$.
It is characterised by its \textit{mean} given by the Bochner integral
\begin{eqnarray*}
m & = & \int_{H}x\, d\gamma(x)
\end{eqnarray*}
and the \textit{covariance operator} $\obsCov:H\rightarrow H$
characterised by the relation
\[
\left\langle Cu,v\right\rangle =\int\left\langle u-m,x\right\rangle \left\langle v-m,x\right\rangle d\gamma(x).
\]
From this it is clear that the covariance operator is positive-definite
and self-adjoint. Moreover, we note that $\obsCov$ is necessarily trace-class
and the Gaussian can be expressed through eigenvalues $\lambda_{k}^{2}$
and the corresponding eigenbasis $\phi_{k}$
\[
\gamma=\mathcal{L}\left(m+\sum_{i=1}^{\infty}\lambda_{k}\phi_{k}\xi_{k}\right)\text{ with}\quad\xi_{k}\overset{\text{i.i.d}}{\sim}\mathcal{N}(0,1).
\]
The \textit{Cameron-Martin space} associated with $\gamma$ is 
\[
H_{\gamma}=\left\{ x\big\vert x=\sum x_{i}\phi_{i}\text{ s.t. }\sum\frac{1}{\lambda_{i}^{2}}x_{i}^{2}<\infty\right\} \subset H
\]
equipped with the inner product 
\[
\left\langle x,y\right\rangle _{\gamma}=\sum\frac{1}{\lambda_{i}^{2}}x_{i}y_{i}
\]
where $x=\sum x_{i}\phi_{i}$ and $y=\sum y_{i}\phi_{i}$.This space
characterises the support as well as the direction such that 
\[
T_{h\star}\gamma\ll\gamma
\]
where $T_{h}$ is the translation operator $T_{h}(x)=x+h$.

We also consider the Hilbert scale $(\mathcal{H}^s,\big\Vert \cdot \big\Vert_s)$ generated by $\gamma$ and the regularity of a draw $\zeta\sim\gamma$ can be expressed as follows.

\begin{lem}
\label{lem:regularityNoise}(\cite{2012arXiv1203.5753A}) Imposing
Assumption \ref{ass:noiseTraceSumm} the following statements hold:
\begin{enumerate}
\item Let $\zeta$ be a white noise, then $\mathbb{E}\Vert\Gamma^{\frac{\sigma}{2}}\zeta\Vert<\infty$
for all $\sigma>\sigma_{0}$. 
\item Let $u\sim\mu_{0}$, then $u\in\mathcal{H}^{1-\sigma}$ $\;\mu_{0}$-a.s.
for every $\sigma>\sigma_{0}$. 
\end{enumerate}
\end{lem}

\section{Change of Variables for the Posterior}

The state of a model can be described in several ways. In this section,
we present the resulting relationship between two different descriptions
of the same model.
\begin{thm}
\label{thm:IndpendentOfDiscription}Suppose $\mathcal{G}_{n}=\mathcal{O}_{n}\circ G$
with $G:(X,\Vert\cdot\Vert_{X})\rightarrow(Y,\Vert\cdot\Vert_{Y})$
and $\mathcal{O}:(Y,\Vert\cdot\Vert_{Y})\rightarrow(Z,\Vert\cdot\Vert_{Z})$.
Furthermore, assume that the posterior $\posterior$($\tilde{\mu}^{y}$)
is well-defined for the forward operator $\mathcal{G}_{n}$($\mathcal{O}_{n}$),
the prior $\prior(da)$ ($\tilde{\mu}_{0}(dp)$) and the noise $\noise\sim\mathcal{N}(0,\Gamma)$.
It is given by 
\begin{eqnarray*}
\frac{d\posterior}{d\prior}(\input) & \propto\exp\left(-\frac{1}{2}\bigl\Vert\opObs(\input)\bigr\Vert_{\obsCov}^{2}+\left\langle y,\opObs(\input)\right\rangle _{\obsCov}\right)\\
\frac{d\tilde{\mu}^{y}}{d\tilde{\prior}}(\pressure) & \propto\exp\left(-\frac{1}{2}\bigl\Vert\mathcal{O}(\pressure)\bigr\Vert_{\obsCov}^{2}+\left\langle y,\mathcal{O}(\pressure)\right\rangle _{\obsCov}\right).
\end{eqnarray*}
In this case $G_{\star}\posterior=\tilde{\mu}^{y}$.
\end{thm}

\begin{proof}
It is sufficient to show that both measures agree on all sets $A\in\mathcal{B}(Y)$

\[
(G_{*}\posterior)(A)=\underset{A}{\int}1dG_{*}\posterior(da).
\]
By the transformation rule
\begin{eqnarray*}
\hspace{-1cm}(G_{*}\posterior)(A) & = & \underset{G^{-1}(A)}{\int}1d\posterior(\input)\text{\text{=\ensuremath{\underset{G^{-1}(A)}{\int}}c\ensuremath{\cdot\exp\left(-\frac{1}{2}\left\Vert \mathcal{O}(G(v\input))-y\right\Vert _{\Gamma}^{2}\right)}d\ensuremath{\prior}(\ensuremath{\input})}}\\
 & = & \underset{A}{\int}c\cdot\exp\left(-\frac{1}{2}\left\Vert \mathcal{O}(v)-y\right\Vert _{\Gamma}^{2}\right)dG_{*}\mu_{0}(v).
\end{eqnarray*}
\end{proof}

\section{\label{sec:Proof-of-Theorem}Proof of Theorem \ref{thm:postConSmallGeneral}}

\begin{proof}[Proof of Theorem \ref{thm:postConSmallGeneral}]

We follow the same steps as in the proof of Theorem \ref{thm:NoTail}
up to Equation (\ref{eq:postConNoTailInterp}) reading

\begin{eqnarray*}
\left|\left\langle \input-\truth,\noise\right\rangle _{1}\right| & \leq K_{n}\left\Vert \input-\truth\right\Vert _{1}^{\HilbertScalesInter}\left\Vert \input-\truth\right\Vert _{s}^{1-\HilbertScalesInter}\mbox{\ref{eq:postConNoTailInterp}}
\end{eqnarray*}
with $\HilbertScalesInter=\frac{s-1-\sigma_{0}-\gamma}{s-1}$. We
now separate the product using Young's inequality with $\frac{1}{p}+\frac{1}{q}=1$

\begin{eqnarray}
\hspace{-2cm}\left|\left\langle \input-\truth,\noise\right\rangle _{1}\right| & \leq\left(K_{n}\left(2^{(\HilbertScalesInter-1)q}qn^{-\HilbertScalesYoungs}\right){}^{-\frac{1}{q}}\left\Vert \input-\truth\right\Vert _{1}^{\HilbertScalesInter}\right)\left(\left(2^{(\HilbertScalesInter-1)q}qn^{-\HilbertScalesYoungs}\right){}^{\frac{1}{q}}\left\Vert \input-\truth\right\Vert _{s}^{1-\HilbertScalesInter}\right)\nonumber \\
 & \leq\tilde{K_{n}}\left\Vert \input-\truth\right\Vert _{1}^{\HilbertScalesInter p}+n^{-\HilbertScalesYoungs}\left(\frac{1}{2}\left\Vert \input-\truth\right\Vert _{s}\right){}^{(1-\HilbertScalesInter)q}\label{eq:postConSmallNoiseCMInner}
\end{eqnarray}
where, for simplicity of notation, we used $\tilde{K}_{n}:=\frac{K_{n}^{p}(2^{-(1-\HilbertScalesInter)q}qn^{-\HilbertScalesYoungs})^{-\frac{p}{q}}}{p}$. \\

\subsubsection*{Lower bound on $\posterior(B_{\epsilon n^{-\kappa}}^{1}(\truth))$:}

The following lower bound on $\posterior(B_{\epsilon n^{-\kappa}}^{1}(\truth))$
is based on Equation (\ref{eq:postConSmallNoiseCMInner})
\begin{eqnarray}
\hspace{-2.5cm}\posterior\hspace{-0.1cm}\left(B_{\epsilon n^{-\kappa}}^{1}(\truth)\right)\nonumber \hspace{-0.15cm} &\ge\posterior\hspace{-0.1cm}\left(B_{\frac{\epsilon}{2}n^{-\kappa}}^{1}(\truth)\cap B_{R}^{s}(0)\right)\ge Z(n,\xi)\prior\left(B_{\frac{\epsilon n^{-\kappa}}{2}}^{1}\left(\truth\right)\cap B_{R}^{s}(0)\right)\nonumber \\
 &\quad \cdot\exp\hspace{-0.1cm}\left[-\frac{n^{1-2\kappa}}{2}\frac{\epsilon^{2}}{4}-n^{\frac{1}{2}-\lambda p\kappa}\tilde{K}_{n}\left(\frac{\epsilon}{2}\right)^{\HilbertScalesInter p}\hspace{-0.2cm}-n^{\frac{1}{2}-\HilbertScalesYoungs}\left[R^{(1-\HilbertScalesInter)q}+\bigl\Vert\truth\bigr\Vert_{s}^{(1-\lambda)q}\right]\right].\label{eq:postConHilbertScaleDeltaBall}
\end{eqnarray}
The term $n^{1-2\kappa}$ has to be dominant in Equation (\ref{eq:postConHilbertScaleDeltaBall})
because the same exponent is appearing in Equation (\ref{eq:postConHilbertScaleDeltaBallComplement})
except for a larger coefficient. Choosing $R=n^{\theta}$ and substituting
the expression for $\tilde{K}_{n}$, this is the case if
\begin{eqnarray}
1-2\kappa & >\frac{1}{2}+\eta\frac{p}{q}-\kappa\lambda p\label{eq:postConSmallRate1}\\
1-2\deltarate & >\frac{1}{2}-\HilbertScalesYoungs+(1-\HilbertScalesInter)q\nRadiusSnorm\label{eq:postConSmallRate2}\\
\log\prior\left(B_{\frac{\epsilon n^{-\kappa}}{2}}^{1}\left(\truth\right)\cap B_{R}^{s}(0)\right) & \gtrsim n^{1-2\kappa}.\label{eq:postConSmallTmp1}
\end{eqnarray}
We need small ball probabilities and the exponential moments of $\prior$
in order to obtain explicit sufficient conditions on $\kappa$. We
first note that

\begin{eqnarray*}
\prior\left(B_{\frac{\epsilon n^{-\kappa}}{2}}^{1}\left(\truth\right)\cap B_{R}^{s}(0)\right) & \ge & \prior\left(B_{\frac{\epsilon n^{-\kappa}}{2}}^{1}\left(\truth\right)\right)-\prior\left(B_{R}^{s}(0)^{c}\right).
\end{eqnarray*}
Equation (\ref{eq:postConSmallTmp1}) holds if 
\begin{eqnarray}
\rho\kappa & < & e\theta\label{eq:postConSmallRate3}\\
 \rho\kappa&<&1-2\kappa.\label{eq:postConSmallrate4}
\end{eqnarray}

\subsubsection*{Upper bound on $\posterior\left(B_{\epsilon n^{-\kappa}}^{1}(\truth)^{c}\right)$:}

We bound $\posterior\left(B_{\epsilon n^{-\kappa}}^{1}\left(\truth\right){}^{c}\right)$
by
\[
\posterior\left(B_{\epsilon n^{-\kappa}}^{1}(\truth)^{c}\right)\leq\posterior\left(B_{\epsilon n^{-\kappa}}^{1}(\truth)^{c}\cap B_{R}^{s}(0)\right)+\posterior\left(B_{\epsilon n^{-\kappa}}^{1}(\truth)^{c}\cap B_{R}^{s}(0)^{c}\right).
\]

\subsubsection*{Upper bound on $\posterior\left(B_{\epsilon n^{-\kappa}}^{1}(\truth)^{c}\cap B_{R}^{s}(0)\right)$:}

We denote by $M_{B_{\epsilon n^{-\kappa}}^{1}(\truth)^{c}\cap B_{R}^{s}(0)}$
the following supremum

\[
\hspace{-2cm}\underset{B_{\epsilon n^{-\kappa}}^{1}(\truth)^{c}\cap B_{R}^{s}(0)}{\sup}-\frac{n}{2}\left\Vert \input-\truth\right\Vert _{1}^{2}+\sqrt{n}\tilde{K}_{n}\left\Vert \input-\truth\right\Vert _{1}^{\HilbertScalesInter p}+n^{\frac{1}{2}-\HilbertScalesYoungs}\left(\frac{1}{2}\left\Vert \input-\truth\right\Vert _{s}\right){}^{(1-\HilbertScalesInter)q}
\]
which is finite if 
\begin{equation}
\lambda p<2.\label{eq:postConSmallRate5}
\end{equation}
The first two summands above can be rewritten as a function $f$ of
$\left\Vert \input-\truth\right\Vert _{1}$ where 
\begin{eqnarray*}
f(d) & =-\frac{n}{2}d^{2}+\sqrt{n}\tilde{K_{n}}d^{\lambda p}.
\end{eqnarray*}
By considering $f^{\prime}$, we see that $f$ is decreasing for $d\ge\bigl(\tilde{K_{n}}\lambda pn^{-\frac{1}{2}}\bigr)^{\lambda p}$.
Thus, for

\begin{eqnarray}
\epsilon n^{-\kappa} & \ge\bigl(\tilde{K_{n}}\lambda pn^{-\frac{1}{2}}\bigr)^{\lambda p}\label{eq:postConHilbertScaleNecDecreasing}
\end{eqnarray}
the following inequality holds
\begin{eqnarray}
\hspace{-2cm}\posterior\left(B_{\epsilon n^{-\kappa}}^{1}(\truth)^{c}\cap B_{R}^{s}(0)\right)\nonumber\\ \hspace{-1.7cm}\leq Z(n,\xi)\exp\left[-\frac{n^{1-2\kappa}}{2}\epsilon^{2}+n^{\frac{1}{2}-\lambda p\kappa}\tilde{K}_{n}\epsilon^{\HilbertScalesInter p}+\right.
 \left.\quad n^{\frac{1}{2}-\HilbertScalesYoungs}\left(R^{(1-\HilbertScalesInter)q}+\bigl\Vert\truth\bigr\Vert_{s}^{(1-\HilbertScalesInter)q}\right)\right].\label{eq:postConHilbertScaleDeltaBallComplement} 
\end{eqnarray}
Then for large $n$, Equation (\ref{eq:postConHilbertScaleNecDecreasing})
is implied by 
\begin{eqnarray}
\left(\eta\frac{p}{q}-\frac{1}{2}\right)\lambda p & <-\kappa.\label{eq:postConSmallRate6}
\end{eqnarray}

\subsubsection*{Upper bound on $\posterior\left(B_{\epsilon n^{-\kappa}}^{1}(\truth)^{c}\cap B_{R}^{s}(0)^{c}\right)$:}

In this section, we bound $\posterior(B_{\epsilon}^{1}(\truth)^{c}\cap B_{R}^{s}(0)^{c})$
using Markov's inequality in combination with the exponential moments
of the prior 
\begin{eqnarray}
\hspace{-1cm} \posterior\left(\exp(f\left\Vert \cdot\right\Vert _{s}^{e})\chi_{B_{\epsilon n^{-\kappa}}^{1}(\truth)^{c}}\right)\leq\underset{B_{\epsilon n^{-\kappa}}^{1}(\truth)^{c}}{\int}C(n,\xi)\exp\left(n^{\frac{1}{2}-\HilbertScalesYoungs}\left\Vert \input\right\Vert _{s}{}^{(1-\HilbertScalesInter)q}\right)\nonumber \\
 \hspace{-1cm}\exp\left(-\frac{n}{2}\left\Vert \input-\truth\right\Vert _{1}^{2}+\sqrt{n}\tilde{K}_{n}\left\Vert \input-\truth\right\Vert _{1}^{\HilbertScalesInter p}+n^{\frac{1}{2}-\HilbertScalesYoungs}\left\Vert \input^{\dagger}\right\Vert _{s}^{(1-\HilbertScalesInter)q}\right)d\prior(\input).\label{eq:postConHilbertScalesPosteriorMoment}
\end{eqnarray}
We denote the term appearing in the exponential in the second line
by $T_{0}$. It can be bounded similar to the upper bound on\textbf{
$\posterior\left(B_{\epsilon n^{-\kappa}}^{1}(\truth)^{c}\cap B_{R}^{s}(0)\right)$}
\begin{eqnarray*}
T_{0} & \leq\mathfrak{U}_{T_{0}}:=-\frac{n^{1-2\kappa}}{2}\epsilon^{2}+n^{\frac{1}{2}-\lambda p\kappa}\tilde{K}_{n}\epsilon^{\HilbertScalesInter p}+n^{\frac{1}{2}-\HilbertScalesYoungs}\left\Vert \input^{\dagger}\right\Vert _{s}^{(1-\HilbertScalesInter)q}.
\end{eqnarray*}
We denote by $\lesssim$ an inequality with a multiplicative constant
not involving $n$ or $\kappa$. In order to get an upper bound for
Equation (\ref{eq:postConHilbertScalesPosteriorMoment}), we bound
the exponential moment by 
\begin{eqnarray*}
\hspace{-2cm} \posterior\left(\exp(f\left\Vert \cdot\right\Vert _{s}^{e})\chi_{B_{\epsilon n^{-\kappa}}^{1}(\truth)^{c}}\right)\lesssim C(n,\xi)\underset{}{\int}\exp\left(n^{\frac{1}{2}-\HilbertScalesYoungs}\left\Vert \input\right\Vert _{s}^{(1-\HilbertScalesInter)q}+f\left\Vert \input\right\Vert _{s}^{e}+\mathfrak{U}_{T_{0}}\right)d\prior(d).
\end{eqnarray*}
Introducing 
\begin{eqnarray*}
g(r) & =n^{\frac{1}{2}-\HilbertScalesYoungs}r^{(1-\HilbertScalesInter)q}+fr^{e},\\
g^{\prime}(r) & =n^{\frac{1}{2}-\HilbertScalesYoungs}(1-\HilbertScalesInter)qr^{(1-\HilbertScalesInter)q-1}+efr^{e-1}
\end{eqnarray*}
and performing an integration by parts, it follows that
\begin{eqnarray*}
\hspace{-2cm}\frac{\posterior(\exp(f\left\Vert \cdot\right\Vert _{s}^{e})\chi_{B_{}^{1}(\truth)^{c}})}{C(n,\xi)} & \lesssim\underset{}{\int}\exp\left(g(\left\Vert \input\right\Vert _{s})+\mathfrak{U}_{T_{0}}\right)d\prior(\input)\\
 & \lesssim\exp(\mathfrak{U}_{T_{0}})\int\left[\int_{0}^{\left\Vert \input\right\Vert _{s}}g^{\prime}(r)\exp(g(r))dr\right]+1d\prior(\input)\\
 & \lesssim\int_{0}^{\infty}g^{\prime}(r)\exp\left(g(r)+\mathfrak{U}_{T_{0}}\right)d\prior\left(\left\Vert \input\right\Vert _{s}>r\right)dr\\
 & \lesssim\int_{0}^{\infty}g^{\prime}(R)\exp\left(n^{\frac{1}{2}-\HilbertScalesYoungs}r^{(1-\HilbertScalesInter)q}-2fr^{e}\right)dr.
\end{eqnarray*}
The above can only be expected to be finite if 
\begin{eqnarray}
(1-\HilbertScalesInter)q & <e.\label{eq:postConSmallRate7}
\end{eqnarray}
Moreover, we assume that $\HilbertScalesYoungs<\frac{1}{2}$ since
otherwise 
\begin{eqnarray*}
\underset{}{\int}\exp\left(n^{\frac{1}{2}-\HilbertScalesYoungs}\left\Vert \input\right\Vert _{s}^{(1-\HilbertScalesInter)q}+f\left\Vert \input\right\Vert _{s}^{e}\right)d\prior(\input) & \lesssim\underset{}{\int}\exp\left(2f\left\Vert \input\right\Vert _{s}^{e}\right)d\prior(\input).
\end{eqnarray*}
In order to achieve an upper bound, we split the term in the exponential
into $T_{1}:=n^{\frac{1}{2}-\HilbertScalesYoungs}r^{(1-\HilbertScalesInter)q}-fr^{e}$
and $T_{2}:=-fr^{e}$. The first term is negative whenever
\begin{eqnarray*}
r & \ge r_{z}:=\left(n^{\frac{1}{2}-\HilbertScalesYoungs}f^{-1}\right)^{\frac{1}{e-(1-\HilbertScalesInter)q}}.
\end{eqnarray*}
For $n$ large enough $r_{z}\ge1$ holds. On the interval $[0,s_{z}]$
an upper bound $\mathfrak{U}_{T_{1}}$ on the maximum value of $T_{1}$
can be derived as follows
\begin{eqnarray}
T_{1}^{\prime} & = & 0\Rightarrow r=\left((1-\HilbertScalesInter)qn^{\frac{1}{2}-\HilbertScalesYoungs}e^{-1}f^{-1}\right)^{\frac{1}{e-(1-\HilbertScalesInter)q}}\nonumber \\
\mathfrak{U}_{T_{1}} & := & \left(\frac{(1-\HilbertScalesInter)q}{ef}\right)^{\frac{1}{e-(1-\HilbertScalesInter)q}}\left(n^{\frac{1}{2}-\HilbertScalesYoungs}\right)^{1+\frac{1}{e-(1-\HilbertScalesInter)q}}.\label{eq:postConHilbetScalesMomentBound}
\end{eqnarray}
Putting everything together gives rise to 
\begin{eqnarray*}
\hspace{-2.5cm} \frac{\posterior\left(\exp(f\left\Vert \cdot\right\Vert _{s}^{e})\chi_{B_{n^{-\kappa}}^{1}(\truth)^{c}}\right)}{C(n,\xi)}&\lesssim\int\left(n^{\frac{1}{2}-\HilbertScalesYoungs}(1-\HilbertScalesInter)qr^{(1-\HilbertScalesInter)q-1}+efr^{e-1}\right)\exp\left(\mathfrak{U}_{T_{1}}+\mathfrak{U}_{T_{0}}\right)dr\\
 &+\int_{r_{z}}^{\infty}\hspace{-0.2cm}\left(n^{\frac{1}{2}-\HilbertScalesYoungs}(1-\HilbertScalesInter)qr^{(1-\HilbertScalesInter)q-1}+efr^{e-1}\right)\exp\left(\mathfrak{U}_{T_{0}}-fr^{e}\right)dr\\
 &\lesssim n^{a}\exp\left(\mathfrak{U}_{T_{1}}+\mathfrak{U}_{T_{0}}\right)
\end{eqnarray*}
for some $a$. Using Markov's inequality, this yields
\begin{eqnarray}
\posterior\left(B_{}^{1}(\truth)^{c}\cap B_{R}^{s}(0)^{c}\right) & \lesssim C(n,\xi)n^{\frac{1}{2}-\HilbertScalesYoungs}\exp\left(\mathfrak{U}_{T_{0}}+\mathfrak{U}_{T_{1}}-fR^{e}\right).\label{eq:postConHilbertScalesLargeBallComp}
\end{eqnarray}
Again substituting $R=n^{\theta}$, this is asymptotically smaller
than $\exp\left(-\frac{n^{1-2\kappa}}{2}\frac{^{2}}{4}\right)$ if

\begin{equation}
\left(\frac{1}{2}-\HilbertScalesYoungs\right)\left(1+\frac{1}{e-(1-\HilbertScalesInter)q}\right)<\max\left(1-2\kappa,\nRadiusSnorm e\right).\label{eq:postConSmallRate8}
\end{equation}
Collecting the inequalities from above, we see that the results follow
by letting $\gamma\rightarrow0$. 

\end{proof}
\section{Normalising Constant of the BRP}\label{sec:normal}
\begin{proof}[Proof of Lemma \ref{lem:BRPnormalising}]
In order to bound $Z(n,\noise)$ in Equation (\ref{eq:identityPosterior}), we rewrite it as $$\mu^{y_n}=Z(n,\noise)\exp\left(-\Phi\right)\prior,$$ where $Z(n,\noise)=\prior(\exp\left(-\Phi\right))$. We bound $-\Phi$ using the Cauchy-Schwarz inequality
\begin{eqnarray*}
-\Phi & \leq & -\frac{1}{2}n\left\Vert \input\right\Vert _{1}^{2}+n\left\Vert \truth\right\Vert _{1}\left\Vert \input\right\Vert _{1}+n^{\frac{1}{2}}\left\langle a,\xi\right\rangle _{1}.
\end{eqnarray*}
The following steps are quite similar to the steps in the proof of the Theorems
\ref{thm:NoTail} and \ref{thm:postConSmallGeneral}. We treat $\left\langle \input,\noise\right\rangle _{1}$
by smoothing $\xi$ at the expense of $a$ 
\begin{eqnarray*}
\left|\left\langle \input,\noise\right\rangle _{1}\right| & \leq & \left|\left\langle \obsCov^{-1+\frac{1-\sigma_{0}-\gamma}{2}}\input,\obsCov^{\frac{\sigma_{0}-1+\gamma}{2}}\noise\right\rangle \right|\\
 & \leq & \left\Vert \input\right\Vert _{1+\sigma_{0}+\gamma}\left\Vert \xi\right\Vert _{1-\sigma_{0}-\gamma}.
\end{eqnarray*}

We use the interpolation inequality for Hilbert scales with $\HilbertScalesInter=\frac{s-1-\sigma_{0}-\gamma}{s-1}$
(see Lemma \ref{lem:HScaleInterpolation}) and H\"older's inequality
with $\frac{1}{p}+\frac{1}{q}=1$ to obtain
\begin{eqnarray*}
\left\Vert \input\right\Vert _{1+\sigma_{0}+\gamma} & \leq & \left\Vert \input\right\Vert _{1}^{\HilbertScalesInter}\left\Vert \input\right\Vert _{s}^{1-\HilbertScalesInter}\leq\frac{\left\Vert \input\right\Vert _{1}^{p\lambda}}{p}+\frac{\left\Vert \input\right\Vert _{1}^{q(1-\lambda)}}{q}.
\end{eqnarray*}
Combining these bounds yields
\begin{eqnarray*}
\hspace{-1cm}-\Phi & \leq & -\frac{1}{2}n\left\Vert \input\right\Vert _{1}^{2}+n\left\Vert \truth\right\Vert _{1}\left\Vert \input\right\Vert _{1}+\frac{\left\Vert \input\right\Vert _{1}^{p\lambda}}{p}\left\Vert \xi\right\Vert _{1-\sigma_{0}-\gamma}+\frac{\left\Vert \input\right\Vert _{1}^{q(1-\lambda)}}{q}\left\Vert \xi\right\Vert _{1-\sigma_{0}-\gamma}.
\end{eqnarray*}
The first three terms are bounded in $a$ because they are dominated by  the first if $\lambda p<2$. This is implied by choosing $q=\frac{2+\gamma}{2-\lambda}.$ Note that $\left\Vert \xi\right\Vert _{1-\sigma_{0}-\gamma}$
is $\mathbb{\mu}_{\noise_n}$-a.s. bounded due to Lemma \ref{lem:regularityNoise}. Thus  $Z(n,q)$ is bounded below if $e>q$. Letting $\gamma\downarrow0$ in $q$ we see that this is the case for  
\begin{eqnarray*}
e & > & \frac{2\sigma_{0}}{s-1+\sigma_{0}}.
\end{eqnarray*}
An upper bound on $Z(n,q)$ follows from a simple lower bound on $-\Phi$ on $B^{1+\sigma+\gamma}_M(0)$ and the prior measure of this set.

%\begin{eqnarray*}
%\left|\left\langle \input-\truth,\noise\right\rangle _{1}\right| & \leq K_{n}\left\Vert \input-\truth\right\Vert _{1}^{\HilbertScalesInter}\left\Vert \input-\truth\right\Vert _{s}^{1-\HilbertScalesInter}\quad  \hspace{4.5cm}&\mbox{(\ref{eq:postConNoTailInterp})}\\
%\left|\left\langle \input-\truth,\noise\right\rangle _{1}\right|  & \leq K_{n}\left\Vert \input-\truth\right\Vert _{1}^{\HilbertScalesInter}\left\Vert \input-\truth\right\Vert _{s}^{1-\HilbertScalesInter} &\mbox{(\ref{eq:postConNoTailInterp})}
%\end{eqnarray*}

\end{proof}
{\small{\bibliographystyle{plain}
\bibliography{./../../../../docear}

\begin{thebibliography}{10}

\bibitem{2012arXiv1203.5753A}
S.~{Agapiou}, S.~{Larsson}, and A.~M. {Stuart}.
\newblock {Posterior Consistency of the Bayesian Approach to Linear Ill-Posed
  Inverse Problems}.
\newblock {\em Stochastic Processes and Applications}, 2013.
\newblock to appear.

\bibitem{severeillposedbay}
S.~{Agapiou}, A.~M. {Stuart}, and Y.-X. {Zhang}.
\newblock {Bayesian Posterior Contraction Rates for Linear Severely Ill-posed
  Inverse Problems}.
\newblock {\em ArXiv e-prints}, October 2012.

\bibitem{alessandrini1988stable}
G.~Alessandrini.
\newblock Stable determination of conductivity by boundary measurements.
\newblock {\em Appl. Anal.}, 27(1-3):153--172, 1988.

\bibitem{gaussianMeasureas}
Vladimir~I. Bogachev.
\newblock {\em {G}aussian measures}, volume~62 of {\em Mathematical Surveys and
  Monographs}.
\newblock American Mathematical Society, Providence, RI, 1998.

\bibitem{1974BorellLogConcave}
C.~Borell.
\newblock Convex measures on locally convex spaces.
\newblock {\em Ark. Mat.}, 12:239--252, 1974.

\bibitem{choi2007posterior}
T.~Choi and M.~J. Schervish.
\newblock On posterior consistency in nonparametric regression problems.
\newblock {\em J. Multivariate Anal.}, 98(10):1969--1987, 2007.

\bibitem{2010SchwabEllipticUQ}
A.~Cohen, R.~A. Devore, and C.~Schwab.
\newblock Convergence rates of best {$N$}-term {G}alerkin approximations for a
  class of elliptic s{PDE}s.
\newblock {\em Found. Comput. Math.}, 10(6):615--646, 2010.

\bibitem{UQellipticPDESchwab}
A.~Cohen, R.~A. Devore, and C.~Schwab.
\newblock Analytic regularity and polynomial approximation of parametric and
  stochastic elliptic {PDE}'s.
\newblock {\em Anal. Appl. (Singap.)}, 9(1):11--47, 2011.

\bibitem{StochEqnInf}
Giuseppe Da~Prato and Jerzy Zabczyk.
\newblock {\em {S}tochastic equations in infinite dimensions}, volume~44 of
  {\em Encyclopedia of Mathematics and its Applications}.
\newblock Cambridge University Press, Cambridge, 1992.

\bibitem{Con3}
M.~Dashti, S.~Harris, and A.~M {Stuart}.
\newblock {B}esov priors for {B}ayesian inverse problems.
\newblock {\em Inverse Probl. Imaging}, 6:183--200, 2012.

\bibitem{2013dashtiMap}
M.~{Dashti}, K.~J.~H. {Law}, A.~M. {Stuart}, and J.~{Voss}.
\newblock {MAP Estimators and Posterior Consistency in Bayesian Nonparametric
  Inverse Problems}.
\newblock {\em ArXiv preprint 1303.4795}, 2013.

\bibitem{UncertaintyElliptic}
M.~Dashti and A.~M. Stuart.
\newblock Uncertainty quantification and weak approximation of an elliptic
  inverse problem.
\newblock {\em SIAM J. Numer. Anal.}, 49:2524--2542, 2011.

\bibitem{devore1998nonlinear}
R.~A. DeVore.
\newblock Nonlinear approximation.
\newblock In {\em Acta numerica, 1998}, volume~7 of {\em Acta Numer.}, pages
  51--150. Cambridge Univ. Press, Cambridge, 1998.

\bibitem{postConIncon}
P.~Diaconis and D.~A. Freedman.
\newblock On inconsistent {B}ayes estimates of location.
\newblock {\em Ann. Statist.}, 14(1):68--87, 1986.

\bibitem{MR829555}
P.~Diaconis and D.~A. Freedman.
\newblock On the consistency of {B}ayes estimates.
\newblock {\em Ann. Statist.}, 14(1):1--67, 1986.
\newblock With a discussion and a rejoinder by the authors.

\bibitem{doob1949application}
J.~L. Doob.
\newblock Application of the theory of martingales.
\newblock {\em Le calcul des probabilites et ses applications}, pages 23--27,
  1949.

\bibitem{inverseProblem}
Heinz~W. Engl, Martin Hanke, and Andreas Neubauer.
\newblock {\em Regularization of inverse problems}, volume 375 of {\em
  Mathematics and its Applications}.
\newblock Kluwer Academic Publishers Group, Dordrecht, 1996.

\bibitem{MR2597943}
Lawrence~C. Evans.
\newblock {\em Partial differential equations}, volume~19 of {\em Graduate
  Studies in Mathematics}.
\newblock American Mathematical Society, Providence, RI, second edition, 2010.

\bibitem{ferraty2006nonparametric}
Fr{\'e}d{\'e}ric Ferraty and Philippe Vieu.
\newblock {\em Nonparametric functional data analysis}.
\newblock Springer Series in Statistics. Springer, New York, 2006.
\newblock Theory and practice.

\bibitem{florens2012regularized}
J.-P. Florens and A.~Simoni.
\newblock Regularized posteriors in linear ill-posed inverse problems.
\newblock {\em Scand. J. Stat.}, 2012.

\bibitem{MR0158483}
D.~A. Freedman.
\newblock On the asymptotic behavior of {B}ayes' estimates in the discrete
  case.
\newblock {\em Ann. Math. Statist.}, 34:1386--1403, 1963.

\bibitem{MR1790007}
S.~Ghosal, J.~K. Ghosh, and A.~W. van~der Vaart.
\newblock Convergence rates of posterior distributions.
\newblock {\em Ann. Statist.}, 28(2):500--531, 2000.

\bibitem{vaartBook}
Subhashis Ghosal and Aad van~der Vaart.
\newblock Fundamentals of nonparametric bayesian inference.
\newblock 2012.

\bibitem{MR1814364}
David Gilbarg and Neil~S. Trudinger.
\newblock {\em Elliptic partial differential equations of second order}.
\newblock Classics in Mathematics. Springer-Verlag, Berlin, 2001.
\newblock Reprint of the 1998 edition.

\bibitem{HairerSPDE}
M.~Hairer.
\newblock {An Introduction to Stochastic PDEs}.
\newblock Lecture Notes, 2009.

\bibitem{nonlinearsampling}
M.~Hairer, A.~M. Stuart, and J.~Voss.
\newblock {A}nalysis of {S}{P}{D}{E}s {A}rising in {P}ath {S}ampling. {P}art 2:
  {T}he {N}onlinear {C}ase.
\newblock {\em Ann. Appl. Probab.}, pages 1657--1706, 2007.

\bibitem{hansen2012inverse}
T.~M. Hansen, K.~S. Cordua, and K.~Mosegaard.
\newblock {Inverse problems with non-trivial priors: Efficient solution through
  sequential Gibbs sampling}.
\newblock {\em Comput. Geosci.}, pages 1--19, 2012.

\bibitem{BayesianNonBook}
Nils~Lid Hjort, Chris Holmes, Peter M{\"u}ller, and Stephen~G. Walker, editors.
\newblock {\em Bayesian nonparametrics}, volume~28 of {\em Cambridge Series in
  Statistical and Probabilistic Mathematics}.
\newblock Cambridge University Press, Cambridge, 2010.

\bibitem{2012arXiv1207.2411H}
V.~H. {Hoang}, C.~{Schwab}, and A.~M. {Stuart}.
\newblock {Sparse MCMC gpc Finite Element Methods for Bayesian Inverse
  Problems}.
\newblock {\em ArXiv preprint 1207.2411}, July 2012.

\bibitem{knapik2011bayesian}
B.~T. Knapik, A.~W. van~der Vaart, and J.~H. van Zanten.
\newblock Bayesian inverse problems with {G}aussian priors.
\newblock {\em Ann. Statist.}, 39(5):2626--2657, 2011.

\bibitem{kuchment2012radon}
P.~Kuchment and G.~Uhlmann.
\newblock {The Radon and X-Ray Transforms}.
\newblock 2012.

\bibitem{smallballsurv}
W.~V. Li and Q.-M. Shao.
\newblock Gaussian processes: inequalities, small ball probabilities and
  applications.
\newblock In {\em Stochastic processes: theory and methods}, volume~19 of {\em
  Handbook of Statist.}, pages 533--597. North-Holland, Amsterdam, 2001.

\bibitem{MR1472736}
M.~A. Lifshits.
\newblock {\em Gaussian random functions}, volume 322 of {\em Mathematics and
  its Applications}.
\newblock Kluwer Academic Publishers, Dordrecht, 1995.

\bibitem{smallBallBib}
M.~A. Lifshits.
\newblock Bibliography of small deviation probabilities, July 2012.

\bibitem{ray2012bayesian}
K.~Ray.
\newblock Bayesian inverse problems with non-conjugate priors.
\newblock {\em arXiv preprint arXiv:1209.6156}, 2012.

\bibitem{MR628945}
G.~R. Richter.
\newblock An inverse problem for the steady state diffusion equation.
\newblock {\em SIAM J. Appl. Math.}, 41(2):210--221, 1981.

\bibitem{MR1670907}
John Roe.
\newblock {\em Elliptic operators, topology and asymptotic methods}, volume 395
  of {\em Pitman Research Notes in Mathematics Series}.
\newblock Longman, Harlow, second edition, 1998.

\bibitem{Con4}
C.~Schwab and A.~M. Stuart.
\newblock Sparse deterministic approximation of {B}ayesian inverse problems.
\newblock {\em Inverse Probl.}, 28(4):045003, 32, 2012.

\bibitem{shen2001rates}
X.~Shen and L.~Wasserman.
\newblock Rates of convergence of posterior distributions.
\newblock {\em Ann. Statist.}, 29(3):687--714, 2001.

\bibitem{MR2652785}
A.~M. Stuart.
\newblock {I}nverse problems: a {{B}}ayesian perspective.
\newblock {\em Acta Numer.}, 19:451--559, 2010.

\bibitem{stuartchinanotes}
A.~M. {Stuart}.
\newblock {The Bayesian Approach To Inverse Problems}.
\newblock {\em ArXiv preprint 1302.6989}, 2013.

\bibitem{triebel2008function}
Hans Triebel.
\newblock {\em Function spaces and wavelets on domains}, volume~7 of {\em EMS
  Tracts in Mathematics}.
\newblock European Mathematical Society (EMS), Z\"urich, 2008.

\bibitem{MR2418663}
A.~W. van~der Vaart and J.~H. van Zanten.
\newblock Rates of contraction of posterior distributions based on {G}aussian
  process priors.
\newblock {\em Ann. Statist.}, 36(3):1435--1463, 2008.

\bibitem{walker2004new}
S.~Walker.
\newblock New approaches to {B}ayesian consistency.
\newblock {\em Ann. Statist.}, 32(5):2028--2043, 2004.

\bibitem{MR1511670}
H.~Weyl.
\newblock Das asymptotische {V}erteilungsgesetz der {E}igenwerte linearer
  partieller {D}ifferentialgleichungen (mit einer {A}nwendung auf die {T}heorie
  der {H}ohlraumstrahlung).
\newblock {\em Math. Ann.}, 71(4):441--479, 1912.

\bibitem{yeh1986review}
William~WG Yeh.
\newblock {Review of parameter identification procedures in groundwater
  hydrology: The inverse problem}.
\newblock {\em {Water Resources Research}}, 22(2):95--108, 1986.

\end{thebibliography}
}}

\end{document}